\documentclass[review]{elsarticle}
\usepackage{amsfonts}
\usepackage{amsmath}
\usepackage{amsthm}
\usepackage{graphicx}
\usepackage[margin=1in]{geometry}
\usepackage{algorithm}
\usepackage[noend]{algpseudocode}

\makeatletter
\def\ps@pprintTitle{%
 \let\@oddhead\@empty
 \let\@evenhead\@empty
 \def\@oddfoot{}%
 \let\@evenfoot\@oddfoot}
\makeatother

\usepackage{hyperref}
\usepackage{multirow}
\usepackage{booktabs}
\usepackage[nolist]{acronym}%
\usepackage{caption} 

\journal{Sustainable Energy, Grids and Networks}

\newtheorem{assump}{Assumption}
\newtheorem{theorem}{Theorem}

\newtheorem*{theorem*}{Theorem}
\algnewcommand\algorithmicinput{\textbf{Indices and Definitions:}}
\algnewcommand\Definitions{\item[\algorithmicinput]}%

\usepackage{xcolor}

\def\thesismode{0}
\begin{document}

\begin{frontmatter}

\title{A Two-Stage Polynomial Approach to Stochastic Optimization of District Heating Networks}
\author[a]{Marc Hohmann\corref{mycorrespondingauthor}}
\author[b]{Joseph Warrington}
\author[b]{John Lygeros}
\address[a]{Urban Energy Systems Group, Empa, Swiss Federal Laboratories for Materials Science and Technology, \"Uberlandstrasse 129, 8600 D\"ubendorf, Switzerland}
\address[b]{Automatic Control Laboratory, ETH Zurich, Physikstrasse 3, 8092 Z\"urich, Switzerland}

\cortext[mycorrespondingauthor]{Corresponding author}

\begin{abstract}
In this paper, we use stochastic polynomial optimization to derive high-performance operating strategies for heating networks with uncertain or variable demand. The heat flow in district heating networks can be regulated by varying the supply temperature, the mass flow rate, or both simultaneously, leading to different operating strategies. The task of choosing the set-points within each strategy that minimize the network losses for a range of demand conditions can be cast as a two-stage stochastic optimization problem with polynomial objective and polynomial constraints. We derive a generalized moment problem (GMP) equivalent to such a two-stage stochastic optimization problem, and describe a hierarchy of moment relaxations approximating the optimal solution of the GMP. Under various network design parameters, we use the method to compute (approximately) optimal strategies when one or both of the mass flow rate and supply temperature for a benchmark heat network. We report that the performance of an optimally-parameterized fixed-temperature variable-mass-flow strategy can approach that of a fully variable strategy.

\end{abstract}

\begin{keyword}
	District heating, Operating strategies, Two-stage stochastic optimization, Generalized moment problem
\end{keyword}
\begin{acronym}
	\acro{ADP}{Approximate Dynamic Programming}
	\acro{CHP}{Combined Heat and Power plant}
	\acro{COP}{Coefficient of Performance}
	\acro{DDP}{Dual Dynammic Programming}
	\acro{DP}{Dynammic Programming}
	\acro{ESMP}{Energy Storage Management Problem}
	\acro{GMP}{Generalized Moment Problem}
	\acro{HP}{Heat Pump}	
	\acro{LP}{Linear Programming}
	\acro{MILP}{Mixed-Integer Linear Programming}
	\acro{MINLP}{Mixed-Integer Nonlinear Programming}
	\acro{NLP}{Nonlinear Programming}
	\acro{NP-hard}{Nondeterministic Polynomial time-hard}
	\acro{OPF}{Optimal Power Flow}
	\acro{PCM}{Phase-Change Materials}
	\acro{RES}{Renewable Energy Sources}
	\acro{SDP}{semidefinite programming}
	\acro{SOS}{Sum-of-Squares}
\end{acronym}

\end{frontmatter}


\if\thesismode1 \chapter{Two-Stage Polynomial Approach to Stochastic Optimization of District Heating Networks}\label{ch:twostage}\fi
\section{Introduction}
Climate change mitigation and the integration of renewable energy sources into the current energy system have sparked interest in the active management of heating systems. Among the various heating technologies in use today, district heating networks have drawn particular attention since they are often used in high density urban areas with significant potential to reduce operating costs and greenhouse gas emissions.

\subsection{Modelling and optimization of district heating networks}

District heating networks can be modelled and operated in various ways \allowbreak\cite{Talebi2016}, but they share some basic characteristics. A change in heat demand can be met either by changing the difference between the supply and return temperatures, the mass flow rate, or both, leading to different operating strategies. Older networks are mostly controlled by varying the supply temperature and keeping the mass flow constant, whereas newer systems tend to have variable mass flow control \cite{Duquette2016}. Each strategy is subject to a trade-off in terms of losses; higher supply and return temperatures lead to increased heat losses, whereas higher mass flow rates increase the hydraulic losses in the pipes. This trade-off has been studied in different contexts. A comparison of strategies for primary networks\footnote{The primary network transports heat from the generators to substations located at buildings or building clusters. The heat distribution among a cluster of buildings or inside a building is referred to as the secondary network.} can be found in \cite{Pirouti2013}. Hydraulic control strategies for the primary and secondary network were optimized in \cite{Jie2015}. New mass flow regulation schemes using pumps were compared to the traditional strategy of controlling the consumer side heat flow using valves in \cite{Kuosa2013}. The performance of district heating networks with multiple sources was studied in \cite{Wang2018}. 

Significant effort has also been invested in simplified models of these systems. The steady-state thermal losses of a network can be modelled as an exponential temperature drop along a pipe segment \cite{Wang2018,AWAD2009, Liu2014,Tang2014}. By replacing the exponential by its first order Taylor approximation, the authors of \cite{AWAD2009} and \cite{Tang2014} obtain a polynomial representation of the pipe output temperature. The hydraulic losses, namely the pressure drop along pipes and substations, are mass flow rate dependent and can be characterised implicitly by the nonlinear Colebrook-White equation \cite{Wang2018}. To simplify this representation, it is often assumed that a pipe segment has a constant coefficient of resistance \cite{AWAD2009,Tang2014,vanderheijde2017}, making the absolute pressure losses quadratic in the mass flow rate. Thus, both types of network loss can be modelled using polynomial functions. 

\subsection{Two-stage stochastic programs}

If the operating strategy keeps a control variable fixed (either temperature or mass flow rate), it is desirable that the fixed choice leads to acceptable performance over a range of demand conditions. Here we define this performance as the expected operational cost incurred by hydraulic and thermal losses with respect to a probability distribution of the heat demand. The problem of determining the optimal set-point of the fixed control variable minimizing the operating cost can be cast as a two-stage stochastic program. In the first stage, the optimal set-point of the fixed control variable is selected. The heat supply satisfying the consumer heat demand is optimized in the second stage by adjusting the remaining control variables.

Two-stage stochastic programs, even with linear constraints, are often intractable. A standard approach is to approximate the program by sampling the space of the uncertain disturbances. This leads to a deterministic problem with a finite number of scenarios. When dealing with convex optimization models, the scenario programs can approximate the stochastic problems with sufficient accuracy \cite{Shapiro2005} and are tractable even for a large number of scenarios \cite{Boyd2010}. There also exist reduction techniques to decrease the computational load of scenario approaches \cite{Dupacova2002}. Alternatively, a scenario-free program for linear models can be obtained by introducing decision rules \cite{Shapiro2005}, e.g.~imposing the restriction that the recourse variables are affine in the disturbance realization, at the cost of optimality. The authors of \cite{Kuhn2011} study the trade-off between tractability and optimality of linear decision rules in stochastic programming. The authors of \cite{Bampou2011} and \cite{georghiou2015} extends this concept to nonlinear decision rules. For polynomial models, interior-point methods or global optimization techniques can be coupled with the scenario method. However, interior-point methods only find local optima and global optimization techniques, for example based on semi-definite relaxations, are computationally demanding, even for deterministic problems \cite{Lasserre2014}. 

\subsection{Summary of contributions}

Using the polynomial representation outlined above, we cast the selection of an optimal set-point as a polynomial two-stage stochastic program with recourse \cite{stoch2003}. We develop an approximation of polynomial two-stage problems based on semi-definite relaxations inspired by \cite{Lasserre2010}. Using duality arguments, we derive a \ac{GMP} that is equivalent to the two-stage problem and provide a sparse hierarchy of \ac{SDP} relaxations, which returns an estimate of the optimal set-point and an expected cost estimate. Since the heat demand statistics are incorporated via moments and not scenarios, the \ac{SDP} relaxations have almost the same computational complexity as the Lasserre hierarchy for standard polynomial optimization \cite{lasserre2001a}. We use our approach to evaluate operating strategies for different network design parameters. Based on data from \cite{rees2011} and \cite{Pirouti2013} for a medium-sized district heating network connected to several clusters of buildings, we generate a number of network design cases. For each case, we are then able to choose optimal set-points in a rigorous manner, and compare the performance of the strategies in terms of hydraulic and thermal losses for a range of operating conditions.

\subsection{Paper structure}

The polynomial two-stage stochastic program is presented in Section \ref{twostagewrecourse}. The \ac{GMP} equivalent to the two-stage stochastic program is developed in \ref{twostagepoly}. The hierarchy of semidefinite relaxations is presented in Section \ref{twostagemoment}. The district heating model and approach to component sizing is described in Section \ref{dhmodel}. The optimization models to determine the set-points for each strategy are presented in Section \ref{sec:opstrat}. We provide the results of the numerical case study in Section \ref{casestudy}, and some concluding remarks are given in Section \ref{twostagepoly_conclusion}.

%
%
%
%
%

\section{Two-stage stochastic program with recourse}\label{twostagewrecourse}
	Consider the following two-stage stochastic program with recourse:
	\begin{equation}\label{eq:prefix1}
		\begin{aligned}
			\rho=\min_{x_1 \in \mathbb{R}^{n_1}}&f_1(x_1)+\mathbf{E}_\varphi(v^*(x_1,y))\\
			\textrm{s.t.} \quad &q_l(x_1) \geq 0, \quad l=1,\ldots,N_q,
		\end{aligned}
	\end{equation}
	where $v^*: \mathbb{R}^{n_1} \times \mathbb{R}^p\rightarrow\mathbb{R}$ is the value function of the second stage problem
	\begin{equation}\label{eq:prefix2}
	\begin{aligned}
	v ^*(x_1,y)=\min_{x_2 \in \mathbb{R}^{n_2}} &f_2(x_1,x_2)\\
	\textrm{s.t.} \quad &h_{i}(x_1,x_2,y) = 0, \quad i=1,\ldots,N_h,\\
	&g_{j}(x_1,x_2,y) \geq 0, \quad j=1,\ldots,N_g,
	\end{aligned}
	\end{equation}
    and $y \in\mathbb{R}^p$ is an exogenous parameter distributed according to a probability measure $\varphi$ on semi-algebraic set $\mathbf{Y}:=\{g_{j}(y)\geq 0, \quad j=N_g+1,\ldots,N'_g\}$. The expected value $\mathbf{E}_\varphi(v^*(x_1,y))=\int_\mathbf{Y}v^*(x_1,y)d\varphi$ is defined with respect to the probability measure $\varphi$. For the time being, we assume that the distribution of $y$ is unknown; below it will be encoded through moments, obtained for example through samples of historical data. The first and second stage decision variables are $x_1 \in \mathbb{R}^{n_1}$ and $x_2 \in \mathbb{R}^{n_2}$, respectively. The realization of $y$ is not known when $x_1$ is chosen. The term \emph{recourse} refers to the fact that $x_2$ is taken after $y$ is realized. The second stage can therefore be interpreted as a parametric optimization with parameters $x_1$ and $y$. Although the focus of the present study is on heat networks, we note that other systems with polynomial constraints, such as alternating-current power systems \cite{Taylor2015a}, can be modelled in the same manner.
    
    For compactness we now define the set of feasible first-stage decisions $x_1$ as the semi-algebraic set
	\begin{equation*}
		\mathbf{K}_1:=\{x_1 \in \mathbb{R}^{n_1}\,|\,q_l(x_1) \geq 0, \,l=1,\ldots,N_q\}\, ,
	\end{equation*}
and the set of all feasible combinations of control decisions $x_1$, $x_2$ and parameter $y$ as
	\begin{equation*}
			\mathbf{K}_2:=\left\{\begin{array}{l}(x_1,x_2,y) \in \mathbb{R}^{n_1}\times\mathbb{R}^{n_2}\times\mathbb{R}^{p}\end{array} \, \left|\begin{array}{l}
			x_1 \in \mathbf{K}_1\text{, }y \in \mathbf{Y} \\h_i(x_1,x_2,y)=0, \,i=1,\ldots,N_h\\
			g_j(x_1,x_2,y) \geq 0, \,j=1,\ldots,N_g \end{array} \right.\right\} \, .
	\end{equation*}
	We make the following additional assumption with regard to two-stage problem \eqref{eq:prefix1}-\eqref{eq:prefix2}:
	\begin{assump}\label{as:poly2}
		Functions $f_1(x_1)$, $f_2(x_2)$, $q_l(x_1)$, $h_i(x_1,x_2,y)$, and $g_j(x_1,x_2,y)$ are polynomials, sets $\mathbf{K}_1$, $\mathbf{K}_2$ and $\mathbf{Y}$ are compact; and for each first-stage decision $x_1\in\mathbf{K}_1$ and parameter value $y\in\mathbf{Y}$, there exists a feasible second-stage decision, i.e., an $\overline{x}_2$ such that $(x_1, \overline{x}_2, y) \in \mathbf{K}_2$.
	\end{assump}
		
		
	For each $x_1$ and $y$, let $x_2^*(x_1,y)$ be the optimal second stage solution. Under Assumption \ref{as:poly2}, the optimal map $x_2^*(x_1,y)$ is measurable \cite[Theorem 2.2]{Lasserre2010}. Thus, we can compute the expected value $\mathbf{E}_\varphi(v^*(x_1,y))$.

\section{A generalized moment problem for two-stage programs with recourse}\label{twostagepoly}	
	We now state the \ac{GMP} \eqref{eq:onestage}, which encodes a two-stage polynomial problem with recourse. In this problem, we optimize over the measures $\mu_1\in\mathcal{M}(\mathbf{K}_1)_+$ and $\mu_2\in\mathcal{M}(\mathbf{K}_2)_+$, which are supported on $\mathbf{K}_1$ and $\mathbf{K}_2$ respectively.\footnote{The notation $\mathcal{M}(\mathbf{C})_+$ denotes the cone of non-negative Borel measures on a semi-algebraic set $\mathbf{C}$.} Measure $\mu_1$ describes the distribution of decision $x_1$, whereas $\mu_2$ describes the \emph{joint} distribution of $x_1$, $x_2$, and $y$. 
	\begin{subequations}\label{eq:onestage}
		\begin{align}
		\rho_{12}:=\min_{\mu_1 \in \mathcal{M}(\mathbf{K}_1)_+,\,\mu_2 \in \mathcal{M}(\mathbf{K}_2)_+} \quad &\int_{\mathbf{K}_1} f_1d\mu_1+\int_{\mathbf{K}_2} f_2d\mu_2 \label{eq:onestageobj}\\
		\textrm{s.~t.} \quad&\int_{\mathbf{K}_1} d\mu_1 =1, \label{eq:onestagepm}\\
		&\pi_{x_1,y}\mu_2=\mu_1\otimes\varphi. \label{eq:onestagenonanticip} 
		\end{align}
	\end{subequations}
The objective \eqref{eq:onestageobj} is the sum of expected values of the first- and second-stage costs when $x_1$ and $(x_1,x_2,y)$ are distributed according to $\mu_1$ and $\mu_2$ respectively. Constraint \eqref{eq:onestagepm} ensures that $\mu_1$ is a valid probability measure (probability distribution), i.e.~that it integrates to $1$. Constraint \eqref{eq:onestagenonanticip} implies that $\mu_2$ is also a probability measure, since $\mu_1$ and $\varphi$ are probability measures. In constraint (\ref{eq:onestagenonanticip}) the operator $\pi_{x_1,y}\mu_2: \mathcal{M}(\mathbf{K}_2)_+\rightarrow\mathcal{M}(\mathbf{K}_1\times\mathbf{Y})_+$ is the projection of $\mu_2$ from $(x_1,x_2,y)$ space onto $(x_1,y)$ space, and $\mu_1\otimes\varphi$ denotes the product of measures $\mu_1$ and $\varphi$.\footnote{Formally, the projection operator is defined as $\pi_{x_1,y}\mu_2(B)=\mu(B\times\mathbb{R}^{n_2}\cap\mathbf{K}_2)$ for all Borel subsets $B$ of $\mathbf{K}_1\times \mathbf{Y}$, and the product measure is defined as $(\mu_1 \times \varphi)(B_1 \times B_2)=\mu_1(B_1) \varphi(B_2)$ for all Borel sets $B_1 \in \mathbf{K}_1$ and $B_2 \in \mathbf{Y}$.} The constraint imposes the probability distributions of $x_1$ and $y$ as marginals of $\mu_2$. The projection constraint is illustrated in Fig.~\ref{fig:two_stage_fig}.
\begin{figure}
	\centering
    \includegraphics[width=0.6\textwidth]{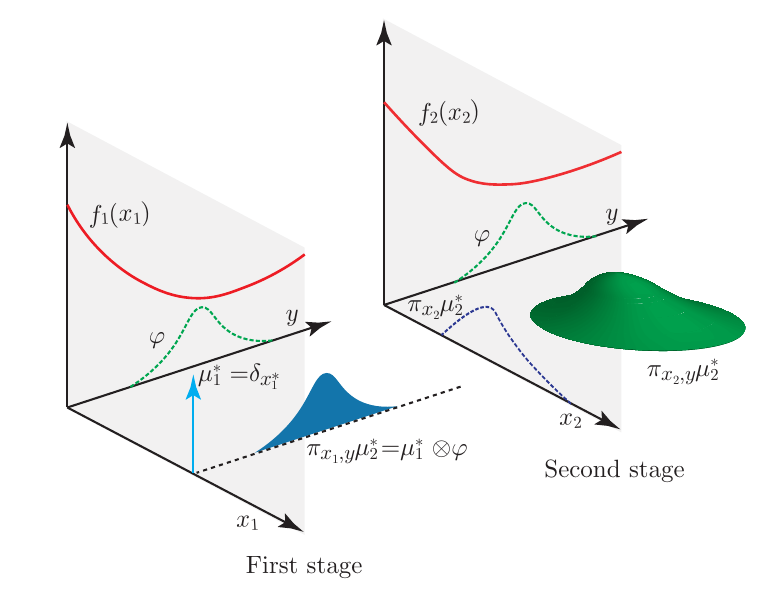}
    \caption[Illustration of the two-stage GMP]{Illustration of the projections of the measure $\mu_2$, (i) onto $(x_1,y)$ space as specified by constraint \eqref{eq:onestagenonanticip} (left plot), and (ii) onto $(x_2,y)$ space (right plot). The objective function \eqref{eq:onestageobj} is defined by integrals over $x_1$ and $x_2$. As described by Theorem \ref{thm:gmp_equiv}, the optimal solution to \eqref{eq:onestage} includes a Dirac solution $\mu_1^* = \delta_{x_1^*}$. As suggested by the \textit{First stage} plot, there will be some value of $x_1$ where it becomes advantageous to concentrate all the mass of $\mu_1$, and there exists an optimal solution featuring a deterministic first-stage decision. However, due to the influence of the second-stage costs, this point need not be the minimum of $f_1(x_1)$.}
    \label{fig:two_stage_fig}
\end{figure}
\begin{theorem} \label{thm:gmp_equiv}
		Two-stage problem (\ref{eq:prefix1})-(\ref{eq:prefix2}) and the \ac{GMP} (\ref{eq:onestage}) are equivalent, in that
\begin{enumerate}
\item[(a)] $\rho_{12}=\rho$, and 
\item[(b)] if $x^*_1$ is an optimal solution of (\ref{eq:prefix1}), then (\ref{eq:onestage}) has an optimal solution $\mu_1^*$ that includes the Dirac measure $\delta_{x^*_1}$.
\end{enumerate}
	\end{theorem}
\if\thesismode 0
The proof is found in \if\thesismode1Appendix \fi\ref{app:gmp_equiv}. In \if\thesismode1Appendix \fi\ref{dualproblem}, we provide the dual of (\ref{eq:onestage}) over bounded continuous functions and show that it approximates the value function $v(x_1,y)$.

\else
	
\fi
	The two-stage generalized moment problem (\ref{eq:onestage}) can be simplified further. 
	Let $\mu$ denote a probability measure supported on $\mathbf{K}_2$, combining the first and second stage. Introducing the operators $\pi_{x_1,y}\mu$ and $\pi_{x_1}\mu$ that project the measure $\mu$ onto $\mathbf{K}_1\times\mathbf{Y}$ and $\mathbf{K}_1$ respectively, we obtain:
	\begin{subequations}\label{eq:genforsdp}
	 	\begin{align}
	 	\rho_{12}=\min_{\mu \in \mathcal{M}(\mathbf{K}_2)_+}\int_{\mathbf{K}_2} &f_1+f_2\,d\mu\\
	 	\textrm{s.t.}\quad&\pi_{x_1,y}\mu=\pi_{x_1}\mu\otimes\varphi.\label{eq:genforsdp_nonant}
	 	\end{align}
	 \end{subequations}
	 The constraint (\ref{eq:genforsdp_nonant}) implies the fixed marginal constraint $\pi_y\mu=\varphi$, where the mapping $\pi_{y}\mu$ projects the measure $\mu$ on $\mathbf{Y}$, because $(\pi_{x_1,y}\mu)(\mathbb{R}^{n_1}\times\mathit{B})=\mu((\mathbb{R}^{n_1}\times\mathbb{R}^{n_2} \times \mathit{B})~\cap \mathbf{K})=\varphi(B)$ for all Borel subsets $B$ of $\mathbf{Y}$. This also makes the constraint $\int_{\mathbf{K}_1} d\mu =1$ redundant since $\varphi(\mathbf{Y})=1$ and $\pi_y\mu=\varphi$ together imply that $\mu(\mathbf{K}_1)=1$. Thus, we can interpret (\ref{eq:genforsdp}) as a \ac{GMP} with fixed marginals and an additional product measure constraint. To sum up, we have reformulated a two-stage stochastic program as a single-stage \ac{GMP}.

	\section{Tractable relaxation of GMP}\label{twostagemoment}
The infinite-dimensional \ac{GMP} (\ref{eq:genforsdp}) can be approximated by an \ac{SDP} relaxation involving a finite number of moments of $\mu$ \cite{Lasserre2014}. \acp{GMP} with constraints on certain marginal distributions have been approximated using \ac{SDP} relaxations in the literature for other purposes \cite{Lasserre2010,Hohmann2018_1}, and the derivation in this section leading to \eqref{eq:twostagesdp} and \eqref{eq:sparsesdp2} is closely related to these. There is a trade-off between the accuracy of the approximation and the computational cost involved, and this is controlled by the choice of \emph{relaxation degree} $k\in\mathbb{N}$. The lowest admissible degree is determined by the degrees of the polynomial functions defining the problem. Specifically, $k$ must satisfy $k\geq \max(\lceil \textrm{deg}\,f/2 \rceil,\max_l d_{q_l},\max_i d_{h_i},\max_j d_{g_j})$, where $d_{q_l}=\lceil \textrm{deg}\,q_l/2 \rceil, l=1,\ldots,N_q$, $d_{h_i}=\lceil \textrm{deg}\,h_i/2 \rceil, i=1,\ldots,N_h$, $d_{g_j}=\lceil \textrm{deg}\,g_j/2 \rceil, j=1,\ldots,N_g'$, ``$\textrm{deg}$'' is the degree of a polynomial, and $\lceil b \rceil$ denotes the ceiling of a real number $b$ (smallest integer greater than or equal to the number $b$). 

\subsection{Moment representation}
In the following, we describe the standard procedure to obtain an equivalent representation of \acp{GMP} such as (\ref{eq:genforsdp}) in terms of an infinite sequence of moments \cite{Lasserre2014} and derive its relaxation of degree $k$.

	First we describe the decision variables and parameters of the moment relaxation. Let $\alpha=(\alpha_1,\ldots,\alpha_{n_1})\in\mathbb{N}^{n_1}$, $\beta=(\beta_1,\ldots,\beta_{n_2})\in\mathbb{N}^{n_2}$ and $\gamma=(\gamma_1,\ldots,\gamma_p)\in\mathbb{N}^p$ be the integer vectors of dimension $n_1$, $n_2$ and $p$ serving as multi-indices and define $m_{\alpha\beta\gamma}$ as moments of the probability measure $\mu$ on $\mathbf{K}_2$ by
	\begin{equation}
	m_{\alpha\beta\gamma}:=\int_{\mathbf{K}_2} x_1^\alpha x_2^\beta y^\gamma  d\mu,
	\end{equation}
	where, following convention, the shorthand symbols $x_1^\alpha:=x_{1,1}^{\alpha_1} x_{1,2}^{\alpha_2}\ldots x_{1,n_1}^{\alpha_{n_1}}$, \mbox{$x_2^\beta:=x_{2,1}^{\beta_1} x_{2,2}^{\beta_2}\ldots x_{2,n_2}^{\beta_{n_2}}$} and $y^\gamma:=y_1^{\gamma_1} y_2^{\gamma_2}\ldots y_p^{\gamma_p}$ are used to represent the monomials. Since $\mu$ is a probability measure, we have $m_{000}=\int_{\mathbf{K}_2}d\mu = 1$. Let $\mathbf{m}$ be a vector containing all the moments ($m_{\alpha\beta\gamma}$) up to degree $2k$ such that $\sum_{t=1}^{n_1} \alpha_{t}+\sum_{t=1}^{n_2} \beta_{t}+\sum_{t=1}^p \gamma_t \leq 2k$. The moments of the exogenous parameter $y$ distributed according to the probability measure $\varphi$ are denoted $z_\gamma=\int_{\mathbf{Y}}y^\gamma d\varphi$, with $z_{0}=1$.
    
 Secondly, we define an operator to represent the objective function in terms of moments of $\mu$. Let $\mathbb{R}_k[x_1,x_2,y]$ be the ring of polynomials of degree at most $k$ in $(x_1,x_2,y)$. For any moment vector $\mathbf{m}$, we can define an associated linear mapping $L_{\mathbf{m}}:\mathbb{R}_k[x_1,x_2,y]\rightarrow\mathbb{R}$,
	\begin{equation}
  u=\sum_{\alpha\beta\gamma}u_{\alpha\beta\gamma}x_1^\alpha x_2^\beta y^\gamma \mapsto L_{\mathbf{m}}(u)=\sum_{\alpha\beta\gamma}u_{\alpha\beta\gamma}m_{\alpha\beta\gamma},
	\end{equation}
for any $u\in\mathbb{R}_k[x_1,x_2,y]$. Using this definition, the objective $\int_{\mathbf{K}_2} f_1+f_2\,d\mu$ can be written equivalently in terms of moments as $L_{\mathbf{m}}(f_1+f_2)$.

Thirdly, we use the operator $L_\mathbf{m}$ to enforce the support of the measure $\mu$, i.e.~$\mu \in \mathcal{M}(\mathbf{K}_2)_+$. Applying $L_\mathbf{m}$ to a polynomial $u$ of degree $k$ defines a positive semi-definite matrix $M_k(\mathbf{m})$, the so called \emph{moment matrix}:
	\begin{equation}
	L_{\mathbf{m}}(u^2)=\mathbf{u}^T M_k(\mathbf{m}) \mathbf{u},
	\end{equation}
	where $\mathbf{u}$ is the vector of coefficients of $u$ and $M_k(\mathbf{m})$ comprises entries $m_{\alpha\beta\gamma}$ of $\mathbf{m}$. The non-negativity of $u^2$ implies that the moment matrix is symmetric positive semidefinite. The positive semidefinite \emph{localizing matrices} $M_{k-d_{q_l}}(q_l\mathbf{m})$ and $M_{k-d_{g_j}}(g_j\mathbf{m})$ enforce inequality constraints $q_l(x_1)\geq 0, l=1,\ldots,N_q$ and $g_j(x_1,x_2,y)\geq 0, j=1,\ldots,N_g'$ and are derived in the same way:
	\begin{equation}
	\begin{aligned}
		L_{\mathbf{m}}(q_lu^2)=\mathbf{u}^T M_{k-d_{q_l}}(q_l\mathbf{m}) \mathbf{u}\geq 0,\\
		L_{\mathbf{m}}(g_ju^2)=\mathbf{u}^T M_{k-d_{g_j}}(g_j\mathbf{m}) \mathbf{u}\geq 0.
	\end{aligned}
	\end{equation}
	The entries of $M_{k-d_{q_l}}(q_l\mathbf{m})$ and $M_{k-d_{g_j}}(g_j\mathbf{m})$ are linear combinations of the moments $m_{\alpha\beta\gamma}$. The entries of the localizing matrices $M_{k-d_{h_i}}(h_i\mathbf{m})$ enforcing equality constraints $h_i(x,y)=0, i=1,\ldots,N_h$, defined as $L_{\mathbf{m}}(h_i u^2)=\mathbf{u}^T M_{k-d_{h_i}}(h_i\mathbf{m})$, must be equal to zero since an equality constraint can be equivalently expressed as two reverse inequalities. This constraint is written as $M_{k-d_{h_i}}(h_i\mathbf{m})=0$. 
    A sequence of moments $\mathbf{m}$ has a representing finite Borel measure $\mu \in \mathcal{M}(\mathbf{K}_2)_+$ if and only if the moment matrix and the localizing matrices are positive semi-definite for all $k\in\mathbb{N}$ \cite[Theorem 3.8]{Lasserre2014}.
    
Finally, we note that since the measure $\mu$ is supported on a compact set, it is completely determined by its (infinite sequence of) moments. Thus, the product measure constraint $\pi_{x_1,y}\mu=\pi_{x_1}\mu\otimes\varphi$ can be written equivalently as an infinite list of moment constraints, $\int_{\mathbf{K}_2}x_1^\alpha y^\gamma d\mu=\int_{\mathbf{K}_2}x^\alpha d\mu\int_{\mathbf{Y}}y^\gamma d\varphi$, or in shorthand $m_{\alpha0\gamma}=m_{\alpha00}z_{\gamma}$, for all $(\alpha,\gamma) \in \mathbb{N}^{n_1}\times\mathbb{N}^{p}$. As we limit ourselves to a finite truncation of these moments, we only enforce these constraints for moments of $\mu$ up to degree $2k$.     
     
     Based on this moment representation of \eqref{eq:genforsdp}, the moment relaxation of degree $k$ is an SDP of the form
	\begin{subequations}\label{eq:twostagesdp}
		\begin{align}
		\rho_k=\min_{\mathbf{m}}\enskip & L_{\mathbf{m}}(f_1+f_2)\\
		\text{s.t.} \quad & M_k(\mathbf{m})\succeq 0,\\
		&M_{k-d_{q_l}}(q_l\mathbf{m})\succeq 0, \quad \, l={1,\ldots N_q,}\label{eq:twostagesdplocal0}\\
		&M_{k-d_{h_i}}(h_i\mathbf{m})=0, \quad i={1,\ldots N_h,}\label{eq:twostagesdplocal1}\\
		&M_{k-d_{g_j}}(g_j\mathbf{m})\succeq 0, \quad j={1,\ldots N_g',}\label{eq:twostagesdplocal2}\\
		&m_{\alpha0\gamma}=m_{\alpha00}z_{\gamma}, \qquad\forall (\alpha,\gamma) \in \mathbb{N}^{n_1}\times\mathbb{N}^{p},\sum_{t=1}^{n_1} \alpha_{t}+\sum_{t=1}^p \gamma_t \leq 2k,\label{eq:nonanticip}
		\end{align}
	\end{subequations} 
	where the notation $A \succeq 0$ indicates that matrix $A$ must be positive semidefinite.
    
\subsection{Sparse representation}

	The dimension  of the semi-definite constraints, given by $\binom{n_1+n_2+p+k}{k} \times \binom{n_1+n_2+p+k}{k}$, grows quickly in $n_1$, $n_2$, and $p$ and represents the primary computational bottleneck for SDP solvers. However, in cases where the underlying problem has a natural sparsity structure (e.g.~a sparse network graph), the computational cost can be significantly reduced for a given relaxation degree $k$. The approach proposed in \cite{Waki2006} exploits the sparsity structure of the set $\mathbf{K}_2$ and the polynomials $f_1$ and $f_2$ to replace each large moment and localizing matrix in (\ref{eq:twostagesdp}) by multiple but significantly smaller ones. We now apply this decomposition to problem \eqref{eq:genforsdp}.
	
	As in \cite{Hohmann2018_1}, define $J$ as the set of all monomials contained in $q_1(x,y),\ldots,q_{N_q}(x_1)$, $h_1(x_1,x_2,y),\allowbreak\ldots,h_{N_h}(x_1,x_2)$, $\allowbreak g_1(x_1,x_2,y),\allowbreak\ldots,g_{N_g'}(y)$, $f_1(x_1)$ and $f_2(x_1,x_2)$. A subset of monomials of $J$ with index $s\in\{1,\ldots,N_I\}$ only involves a subset $I_s$ of the variables $\{x_{1,1},\ldots,x_{1,n_1},x_{2,1},\ldots,x_{2,n_2},y_1,\ldots,y_p\}$. Let $n_{1,s}$, $n_{2,s}$ and $p_s$ be the cardinality of $I_s$ with respect to $x_1$, $x_2$ and $y$. Define the sets $\mathcal{Q}_{s} \subset \{1,\ldots,N_q\}$, $\mathcal{H}_{s} \subset \{1,\ldots,N_h\}$ and $\mathcal{G}_{s}\subset\{1,\ldots,N_g'\}$ as the sets of indices of constraints $q_1(x_1),\ldots,q_{N_q}(x_1)$;$\allowbreak h_1(x_1,x_2,y),\allowbreak\ldots,h_{N_h}(x_1,x_2,y)$; and $\allowbreak g_1(x_1,x_2,y),\ldots,g_{N_g'}(y)$ that include at least one variable of $I_s$.
	
	If the collection $\{I_1,\ldots,I_{N_I}\}$ satisfies the \emph{running intersection property} defined by:
	\begin{equation}\label{eq:RIP}
	\text{For each $s=1,\ldots,N_I-1$,}\quad I_{s+1}\cap\bigcup_{t=1}^s I_t\subseteq I_z \quad \text{for some}\enskip z\leq s,
	\end{equation}
	a so-called \emph{multi-measures} moment problem equivalent to (\ref{eq:genforsdp}) can formulated by virtue of \cite[Theorem 4.6]{Lasserre2014}. Its corresponding SDP relaxation at level $k$ is given by:
	\begin{subequations}\label{eq:sparsesdp2}
		\begin{align}
		\rho^{\rm sp}_k =\min_{\mathbf{m}}\quad & L_{\mathbf{m}}(f_1)+L_{\mathbf{m}}(f_2)\label{eq:sparsesdpobj2}\\
		\text{s.t.}\quad & M_k(\mathbf{m},I_s)\succeq 0\label{eq:sparsesdppsd2},\hspace{2.15cm} s=1,\ldots,N_I,\\
		&M_{k-d_{q_l}}(q_l\mathbf{m},I_s)\succeq 0,\hspace{1.3cm}\forall l\in\mathcal{Q}_s,\,s=1,\ldots,N_I,\label{eq:sparsesdplocalq2}\\
		&M_{k-d_{h_i}}(h_i\mathbf{m},I_s)=0,\hspace{1.2cm}\forall i\in \mathcal{H}_s,\,s=1,\ldots,N_I,\\
		&M_{k-d_{g_j}}(g_j\mathbf{m},I_s)\succeq0,\hspace{1.2cm}\forall j\in\mathcal{G}_s,\,s=1,\ldots,N_I,\label{eq:sparsesdplocal2}\\
		&m_{\alpha 0\gamma}(I_s)=m_{\alpha 0 0}(I_s)z_{\gamma}(I_s),\,\hspace{0.2cm}\forall (\alpha,\gamma) \in \mathbb{N}^{n_{1,s}}\times\mathbb{N}^{p_s},\nonumber \\
        & \hspace{4.5cm}\sum_{t=1}^{n_{1,s}} \alpha_{t}+\sum_{t=1}^{p_s} \gamma_t \leq 2k,\, s=1,\ldots,N_I,\label{eq:nonanticip2}
		\end{align}
	\end{subequations}
	where $m_{\alpha 0\gamma}(I_s)$ are the moments of the first stage variables and exogenous parameters indexed by $I_s$, $M_k(\mathbf{m},I_s)$ are moment matrices constructed from first and second stage variables and exogenous parameters indexed by $I_s$; and $M_{k-d_{q_l}}(q_l\mathbf{m},I_s)$, $M_{k-d_{h_i}}(h_i\mathbf{m},I_s)$ and $M_{k-d_{g_j}}(g_j\mathbf{m},I_s)$ are localizing matrices for constraints indexed by $\mathcal{Q}_s$, $\mathcal{H}_s$ and $\mathcal{G}_s$. The product measure constraint is imposed through the moments indexed by $I_s$ in (\ref{eq:nonanticip2}).
	If the collection $\{I_1,\ldots,I_{N_I}\}$ satisfies (\ref{eq:RIP}), the sparse SDP relaxation (\ref{eq:sparsesdp2}) converges to the optimal solution of $(\ref{eq:genforsdp})$ \cite[Theorem 4.7]{Lasserre2014}, i.e.~$\lim_{k\rightarrow\infty} \rho^{\rm sp}_k=\rho_{12}$. We refer to \cite{Waki2006} for an efficient method to identify a collection $\{I_1,\ldots,I_{N_I}\}$ satisfying (\ref{eq:RIP}). The smaller the cardinalities of the subsets $I_s$, the lower the dimensions of the moment matrices given by $\binom{n_{1,s}+n_{2,s}+p_s+k}{k} \times \binom{n_{1,s}+n_{2,s}+p_s+k}{k}$ at level $k$, and the smaller the computational bottleneck.

\section{District heating network model}\label{dhmodel}
	
	In this section, we describe the steady-state model of a district heating network and its polynomial approximation. The objective is to model accurately the steady-state hydraulic and thermal losses of the system.\footnote{Although we are interested in operating strategies for varying demand levels, we assume that demand variation is slow compared to the pipe flow dynamics, such that the steady-state assumption is still valid over all demand levels.} For this purpose, we only model the primary network and assume that the secondary networks, consisting of the internal heat distribution to a building or cluster of buildings, are controlled separately. We describe the topology of a district heating network using a set of mixing nodes $\mathcal{N}_M=\{1,\ldots,N_M\}$ that are connected by a set of piping branches $\mathcal{B}_P=\{1,\ldots,N_P\}$, load branches $\mathcal{B}_L=\{1,\ldots,N_L\}$ and generator branches $\mathcal{B}_G=\{1,\ldots,N_G\}$. A network schematic illustrating the network components and the hydraulic and thermal modelling variables is shown in Fig.~\ref{fig:network_schematic}.
\begin{figure}
	\centering
	\includegraphics[scale=0.5]{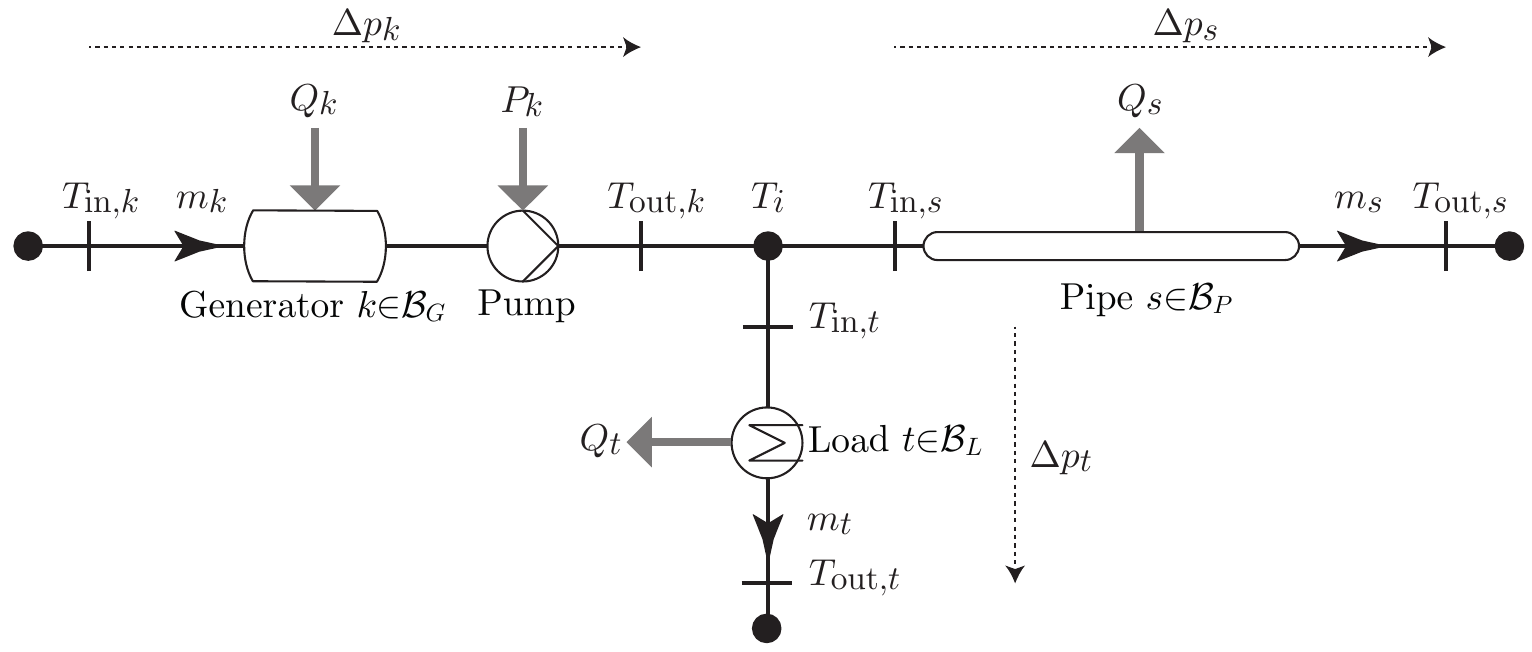}\caption[District heating network components and modelling variables]{Illustration of a generator branch $k\in\mathcal{B}_G$, a pipe branch $s\in\mathcal{B}_P$ and a load branch $t\in\mathcal{B}_L$ connected to a node $i$. For a branch $j\in\mathcal{B}$, $T_{\textrm{in},j}$ and $T_{\textrm{out},j}$ denote inlet and outlet temperatures, $m_j$ the mass flow rate, $\Delta p_j$ the pressure difference and $Q_j$ the thermal power exchange. $P_k$ denotes the pumping power and $T_i$ the mixed temperature at node $i$.}\label{fig:network_schematic}
\end{figure}
\subsection{Hydraulic model}

	We start by describing the hydraulic aspects of the steady-state model used in \cite{Liu2013a}. Let $m_{j}$ denote the mass flow in units [kg/s] in a pipe, load or generator branch indexed by $j\in\mathcal{B}$, where $\mathcal{B}:=\mathcal{B}_P\cup\mathcal{B}_L\cup\mathcal{B}_G$. For each branch, we define a nominal flow direction and the mass flow $m_{j}$ as positive when the flow coincides with the nominal flow direction and negative otherwise. For a node $i\in\mathcal{N}_m$, let $\mathcal{B}^+_{i}$ denote the branches flowing into node $i$, $\mathcal{B}^-_{i}$ the branches flowing out of node $i$. Mass flow conservation requires that the sum of incoming flows be equal to the sum of outgoing flows at each node:
	\begin{equation}\label{eq:massflowconservation}
		\sum_{j\in\mathcal{B}^+_{i}} m_{j}-\sum_{j\in\mathcal{B}^-_{i}} m_{j}=0,\quad\forall i\in\mathcal{N}_M.
	\end{equation}
	

	The pressure drop $\Delta p_s$ in [Pa] along a pipe segment $s\in\mathcal{B}_P$ is approximated as
	\begin{equation}\label{eq:pressuredrop_fitted}
	\Delta p_s=a_{s}|m_{s}|m_{s}+b_{s}m_{s},\quad\forall s\in\mathcal{B}_{P},
	\end{equation}
	where $a_{s}\in\mathbb{R}$ and $b_{s}\in\mathbb{R}$ are friction coefficients. The hydraulic equations approximated by (\ref{eq:pressuredrop_fitted}) are detailed in \if\thesismode1Appendix \fi\ref{app:pressureloss}.
	
	The pressure drop $\Delta p_t$ in [Pa] along a consumer branch $t\in\mathcal{B}_L$ is modelled as
	 \begin{equation}\label{eq:pressuredrop_substation}
	 	\Delta p_t\geq\kappa_{t}m_{t}^2,\quad\forall t\in\mathcal{B}_{L},
	 \end{equation}
	where the coefficient $\kappa_{t}$ is calculated so that the maximum pressure drop at maximum mass flow is $50$kPa \cite{Pirouti2013}. The slack in \eqref{eq:pressuredrop_substation} is the pressure drop over a valve installed in the branch. 
    
Let $\Delta p_{k}$ denote the pressure difference in [Pa] of the generator branch $k\in\mathcal{B}_G$. The nodes $\mathcal{N}_M$ and branches $\mathcal{B}$ form a hydraulic circuit, consisting of $N_{\text{loop}}$ hydraulic loops indexed by $l\in\{1,\ldots,N_{\text{loop}}\}$. Let $\mathcal{B}_{P,l}, \mathcal{B}_{L,l}$ and $\mathcal{B}_{G,l}$ define the sets of pipe, load and generator branches in the loop $l$.  Analogously to Kirchhoff's voltage law, the sum of pressure differences over all branches forming a closed hydraulic loop must be zero:
	\begin{equation}\label{eq:hydraulicloop}
		\sum_{k\in\mathcal{B}_{G,l}} \Delta p_{k}-\sum_{s\in\mathcal{B}_{P,l}} \Delta p_{s}-\sum_{t\in\mathcal{B}_{L,l}} \Delta p_{t}=0,\quad l=1,\ldots,N_{\text{loop}}.
	\end{equation}
    
	 The electric pumping power in [W] is given by \cite{Pirouti2013}
	\begin{equation}\label{eq:pumpingpowereq}
	P_{k}=\frac{\Delta p_{k} m_{k}}{\nu_{\rm pump} \rho} , \quad\forall k\in\mathcal{B}_{G},
	\end{equation}
	where $\nu_{\rm pump}$ is the pump efficiency and $\rho$ the density of water in [kg/m$^3$].

\subsection{Thermal model}

	The nodal temperatures are determined by the thermal model of the district heating network. Let $T_{\textrm{out},s}$ and $T_{\textrm{in},s}$ denote the outlet and inlet temperature of a pipe branch in [$^\circ$C]. We approximate the temperature loss in a pipe segment by
	\begin{equation}\label{eq:thermal_approx}
	T_{\text{out},s}m_{s}=(T_{\text{in},s}-T_{\text{amb}})(c_s m_s-d_s)+T_{\text{amb}}m_{s},
	\end{equation}
    where $c_s\in\mathbb{R}$ and $d_s\in\mathbb{R}$ are loss coefficients derived in \if\thesismode1Appendix \fi\ref{app:thermalloss}.
    
	
	A mixing node $i\in\mathcal{N}_M$ can have multiple inflows with different temperatures. Assuming perfect mixing and applying the law of heat flow conservation, the mixed temperature $T_i$ is governed by:
	 \begin{equation}\label{eq:mixing}
	 \begin{aligned}
			T_{i}\sum_{j\in\mathcal{B}^+_{i}}m_{j}=\sum_{j\in\mathcal{B}^+_{i}}T_{\text{out},j}m_{j},\quad\forall i\in\mathcal{N}_M.
	 \end{aligned}
	\end{equation}
	The temperature of flows leaving node $i$ into links $j$ is equal to the mixed temperature $T_i$:
	\begin{equation}\label{eq:outflows}
	\begin{aligned}
		T_{\textrm{in},j}=T_i,\quad\forall j\in\mathcal{B}^-_i.
	\end{aligned}
	\end{equation}
	Let $T_{\text{in},t}$ and $T_{\text{out},t}$ denote the supply and return temperature of the load branch in [$^\circ$C]. The heat flow exchange $Q_{t}$ in [W] in the load branch $t\in\mathcal{B}_L$ is given by:
	\begin{equation}\label{eq:thermal_load}
	Q_{t}=c_p m_{t} (T_{\text{in},t}-T_{\text{out},t}).
	\end{equation}
	Let $T_{\text{out},k}$ and $T_{\text{in},k}$ denote the supply and return temperature of the generator branch in [$^\circ$C]. The generated heat $Q_{k}$ in [W] in generator branch $k\in\mathcal{B}_G$ is given by:
	\begin{equation}\label{eq:thermal_gen}
	Q_{k}=c_p m_{k} (T_{\text{out},k}-T_{\text{in},k}).
	\end{equation}
The thermal loss $Q_{s}$ in [W] of a pipe branch $s\in\mathcal{B}_P$ calculated as
	\begin{equation}\label{eq:pipe_loss}
	Q_{s}=c_p m_{s}(T_{\text{in},s}-T_{\text{out},s}).
	\end{equation}
\subsection{Operating constraints}

	Most equipment, in particular pipes, pumps and generators, will degrade unless certain operating conditions are maintained. For this reason, we impose technical bounds on temperature, mass flow and pressure variables:
	\begin{subequations}\label{eq:syslimits}
	\begin{align}
		&\underline{m}_j\leq m_j \leq \overline{m}_j, \quad\forall j\in\mathcal{B},\label{eq:syslimitsa}\\	
		&\Delta\underline{p}_j\leq \Delta p_j \leq \Delta\overline{p}_j, \quad\forall j\in\mathcal{B},\label{eq:syslimitsb}\\
		&\underline{T}_i\leq T_i \leq \overline{T}_i, \quad\forall i\in\mathcal{N}_M,\label{eq:syslimitsc}\\
		&\underline{T}_k\leq T_{\textrm{out},k} \leq \overline{T}_k, \quad\forall k\in\mathcal{B}_G,\label{eq:syslimitsd}
	\end{align}
	\end{subequations}
	
	\subsection{Component sizing}\label{designcases}

	The operational characteristics of district heating networks depend not only on the operating strategy but also on the network design. Therefore, we conduct an analysis for multiple networks, for which we briefly describe the design procedure in this section. The topology and pipe lengths of the primary network as well as the maximum heat demand in the load branches are assumed to be given. In \cite{Pirouti2013}, various design cases are studied with regard to investment and operating cost. Using the same design approach, we have the following design parameters:
	\begin{itemize}
		\item Design supply and return temperatures $\hat{T}_{\textrm{in},s}$ and $\hat{T}_{\textrm{out},s}$ of the load branch in [$^\circ$C]
		\item Target pipe pressure loss $\Delta \hat{p}_{s}$ in [Pa/m]
	\end{itemize}
	There is a wide range of target pressure loss values used to design networks. A summary can be found in \cite[Table 1]{Pirouti2013}. We can calculate the maximum load mass flow rate using the design supply and return temperature, and the maximum heat demand $Q_{\textrm{max},s}$ of each load branch \cite{Pirouti2013}:
	\begin{equation}
			m_{\textrm{max},s}=\frac{Q_{\textrm{max},s}}{c_p(\hat{T}_{\textrm{in},s}-\hat{T}_{\textrm{out},s})}.
	\end{equation}
	We find the maximum mass flow in each pipe segment by solving the linear system of equations (\ref{eq:massflowconservation})\footnote{If there are multiple generators, one generator can be defined as the \emph{Slack}, balancing (\ref{eq:massflowconservation}), while the mass flow in the remaining generator branches is fixed to the maximum.}. Given the target pressure loss, we can solve for the diameter using the following expression from \cite{Pirouti2013}
	\begin{equation}\label{eq:diameterdesign}
	D_{s}=\sqrt[5]{\frac{8m_{\textrm{max},s}^2 f_{s}}{(\frac{\Delta \hat{p}_{s}}{L_{s}})\rho\pi^2}} \, ,
	\end{equation}
	and equation (\ref{eq:friction}) in \if\thesismode1Appendix \fi\ref{app:pressureloss}.

\section{Optimal operating strategies} \label{sec:opstrat}

	When supplying heat to consumers, the heat can be regulated using two control variables: supply temperature and mass flow \cite{Gustafsson2010}. This leads to the following three strategies \cite{Pirouti2013}:
	\begin{itemize}
		\item Constant supply temperature, variable mass flow (CT-VF)
		\item Variable supply temperature, constant mass flow (VT-CF)
		\item Variable supply temperature, variable mass flow (VT-VF)
	\end{itemize}
	Most older district heating networks are controlled using VT-CF, whereas newer networks use CT-VF \cite{Duquette2016}. The VT-VF is used here as a benchmark case since it offers the most degrees of freedom. The supply temperature is controlled at the output of each thermal generation unit. The VT-VF and CT-VF strategies require variable-speed pumps, whereas constant-speed pumps are sufficient for the VT-CF strategy.
  
  To evaluate the operational strategies, we compare them in terms of cost incurred by heat and hydraulic losses for a range of heating demand conditions. Since the pumps are electrically-powered and the thermal losses are compensated by burning fuel, the operating cost is a function of the electricity price $c_{\rm elec}$ and the fuel price $c_{\rm fuel}$. The heat demand $Q_{\rm heat}\in\mathbb{R}^{N_L}$, consisting of the individual heat loads $Q_t$, is assumed to follow a known probability distribution $\varphi$ on a known compact set $\mathbf{Y}_L$. For each strategy, we state an optimization problem that determines the operating set-points, namely supply temperature and mass flow, minimizing expected heat and hydraulic losses with respect to $\varphi$. The decision variables of this problem are inlet and outlet temperatures $T_{\textrm{in},j}$ and $T_{\textrm{out},j}$, mass flow rates $m_j$ and pressure drops $\Delta p_j$ for all branches $j$ in $\mathcal{B}$; and nodal temperatures $T_i$ for all nodes $i$ in $\mathcal{N}_M$.
  
  The network constraints described in Section \ref{dhmodel} cannot be directly integrated into the polynomial optimization framework \eqref{eq:prefix1}-\eqref{eq:prefix2}. Whereas the polynomial approximation (\ref{eq:thermal_approx}) is valid for positive and negative mass flow values, we need to determine the flow direction in advance to remove the absolute value operator in (\ref{eq:pressuredrop_fitted}) and to establish equations (\ref{eq:mixing}) and (\ref{eq:outflows}). This can be done easily for networks with tree topologies and a small number of generators, as in the example below. For meshed networks, we can fix the flow directions in each branch in advance at the cost of reducing the feasible space of operating decisions. Alternatively, we can model different inflow/outflow conditions of each mixing node using polynomial equations, although this would introduce additional computational complexity.

The optimal operational set-point for each load occurrence $Q_\textrm{heat}\in\mathbf{Y}_L$ is computed in the case of the VT-VF strategy using
	\begin{equation}\label{eq:vtvfoptim}
	\begin{aligned}
	\rho(Q_{\rm heat})=&\min_{\{T_{i}\}_{i\in\mathcal{N}_M},\{T_{\textrm{in},j},T_{\textrm{out},j},m_{j},\Delta p_j\}_{j\in\mathcal{B}} }\sum_{k\in\mathcal{B}_G}\frac{c_{\rm fuel}}{\nu_{\rm fuel}}Q_k+\sum_{k\in\mathcal{B}_G}c_{\rm elec}P_k\\
	&\text{s.t. (\ref{eq:massflowconservation})-(\ref{eq:thermal_gen}) and (\ref{eq:syslimits})},
	\end{aligned}
	\end{equation}
	where $\nu_{\rm fuel}$ is the fuel to heat efficiency. The expected cost $\mathbf{E}_{\varphi}(\rho(Q_{\rm heat}))$ over all instances of \eqref{eq:vtvfoptim} serves as a reference value to evaluate the performance of the CT-VF and VT-CF strategies. The computation of $\mathbf{E}_{\varphi}(\rho(Q_{\rm heat}))$ is detailed in Section \ref{casestudy}.

The optimal operational set-points of the VT-CF and CT-VF strategies are a solution to a two-stage stochastic program (\ref{eq:prefix1})-(\ref{eq:prefix2}) cast in two different ways. For the VT-CF strategy, the optimal pump mass flows are determined in the first stage:
	\begin{equation}\label{eq:vtcfoptim_1}
	\begin{aligned}
	\min_{\{m_k\}_{k\in\mathcal{B}_G}} &\mathbf{E}_\varphi(v(\{m_k\}_{k\in\mathcal{B}_G},Q_{\rm heat})) \\
	&\text{s.t. (\ref{eq:syslimitsa})}
	\end{aligned}
	\end{equation}
	where $v(\{m_k\}_{k\in\mathcal{B}_G},Q_{\rm heat})$ is the value function of the second stage problem:
	\begin{equation}\label{eq:vtcfoptim_2}
	\begin{aligned}
	v(\{m_k\}_{k\in\mathcal{B}_G},Q_{\rm heat})=&\min_{\substack{\{T_{i}\}_{i\in\mathcal{N}_M},\\ \{{T_{\textrm{in},j},T_{\textrm{out},j},\Delta p_j}\}_{j\in\mathcal{B}},\\ \{m_{j}\}_{j\in\mathcal{B}_P\cup\mathcal{B}_L}}} \sum_{k\in\mathcal{B}_G}c_{\rm elec}P_k+\sum_{k\in\mathcal{B}_G}\frac{c_{\rm fuel}}{\nu_{\rm fuel}}Q_k\\
	&\text{s.t. (\ref{eq:massflowconservation})-(\ref{eq:thermal_gen}) and (\ref{eq:syslimitsb})-(\ref{eq:syslimitsd})}.
	\end{aligned}
	\end{equation}
	For the CT-VF strategy, the optimal supply temperatures of the thermal generation units are determined in first stage:
	\begin{equation}\label{eq:ctvfoptim_1}
	\begin{aligned}
	\min_{\{T_{\textrm{out},k}\}_{k\in\mathcal{B}_G}} &\mathbf{E}_\varphi(v(\{T_{\textrm{out},k}\}_{k\in\mathcal{B}_G},Q_{\rm heat})) \\
	&\text{s.t. (\ref{eq:syslimitsd})},
	\end{aligned}
	\end{equation}
	where $v(\{T_{\textrm{out},k}\}_{k\in\mathcal{B}_G},Q_{\rm heat})$ is the value function of the second stage problem:
	\begin{equation}\label{eq:ctvfoptim_2}
	\begin{aligned}
	v(\{T_{\textrm{out},k}\}_{k\in\mathcal{B}_G},Q_{\rm heat})=&\min_{\substack{\{T_{i}\}_{i\in\mathcal{N}_M},\\ \{{T_{\textrm{in},j},m_j,\Delta p_j}\}_{j\in\mathcal{B}},\\ \{T_{\textrm{out},j}\}_{j\in\mathcal{B}_P\cup\mathcal{B}_L}}}\sum_{k\in\mathcal{B}_G}c_{\rm elec}P_k+ \sum_{k\in\mathcal{B}_G}\frac{c_{\rm fuel}}{\nu_{\rm fuel}}Q_k\\
	&\text{s.t. (\ref{eq:massflowconservation})-(\ref{eq:thermal_gen}) and (\ref{eq:syslimitsa})-(\ref{eq:syslimitsc})}.
	\end{aligned}
	\end{equation}	
	We use the \ac{SDP} relaxations described in Section \ref{twostagemoment} to approximate (\ref{eq:vtcfoptim_1})-(\ref{eq:vtcfoptim_2}) and (\ref{eq:ctvfoptim_1})-(\ref{eq:ctvfoptim_2}). The heat produced in the generator branches is the sum of all heating loads and pipe losses, hence minimizing the generator cost is equivalent to minimizing the thermal lossses. We compare the strategies in terms of the expected cost incurred by thermal losses $\mathbf{E}_\varphi (\sum_{s\in\mathcal{B}_P}\frac{c_{\rm fuel}}{\nu_{\rm fuel}}Q_s$), where $Q_s$ is the thermal loss in a pipe as defined in \eqref{eq:pipe_loss}, and the expected cost incurred by hydraulic losses $\mathbf{E}_\varphi (\sum_{k\in\mathcal{B}_G}c_{\rm elec}P_k)$.
	
\section{Numerical results}\label{casestudy}
We evaluate the operating strategies on the network presented in \cite{rees2011}. This case was studied in \cite{Pirouti2013} in terms of operational strategies. In contrast to \cite{Pirouti2013}, we study the VT-VF strategy using techniques approximating the global optimum, and rigorously optimize the set-points of CT-VF and VT-CF strategies. We follow the design procedure of Section \ref{designcases} to determine pipe diameters using a range of design temperatures and target pressure loss values from \cite{Pirouti2013}. The topology of the supply and return network is identical and shown in Fig. \ref{fig:networktopology}. There is a single generator connected to node $1$ and the generator branch is equipped with a pump. 
\begin{figure}
	\centering
	\includegraphics[trim={0 0cm 0 0},clip,scale=1]{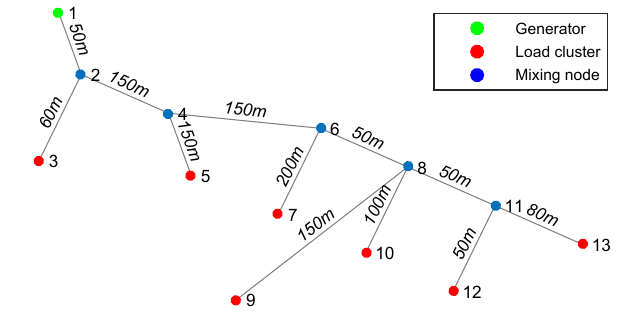}
	\caption[District heating network topology]{District heating network topology, adapted from \cite{rees2011}, representing the supply and return network}\label{fig:networktopology}
\end{figure}

Tables \ref{table:branchdata} and \ref{table:loaddata} summarize the network parameters from \cite{Pirouti2013} and \cite{rees2011} for branches and load nodes. We linearly approximate the values given in Table \ref{table:heattr} from \cite{Liu2013a} to compute the heat transfer coefficients for different pipe designs. Additional problem input data is given in Table \ref{table:inputdataobj}. 

\begin{table}
	\caption[Distrct heating network data: Pipe length and roughness]{Network data \cite{Pirouti2013,rees2011}: Pipe length $L_s$ and pipe roughness $\epsilon_s$}
	\label{table:branchdata}
	\begin{center}
		\begin{tabular}{ l l l l}
			From node& To node & $L_s$ [m] & $\epsilon_s$ [mm]\\ 
			\midrule
			1&2&50&0.4\\
			2&3&60&0.4\\
			2&4&150&0.4\\
			4&5&150&0.4\\
			4&6&150&0.4\\
			6&7&200&0.4\\
			6&8&50&0.4\\
			8&9&150&0.4\\
			8&10&100&0.4\\
			8&11&50&0.4\\
			11&12&50&0.4\\
			11&13&80&0.4\\
			\midrule
		\end{tabular}
	\end{center}
\end{table}
\begin{table}
	\caption[Distrct heating network data: Maximum load]{Network data \cite{rees2011}: Maximum load $Q_{\textrm{max},t}$ connected to node}
	\label{table:loaddata}
	\begin{center}
		\begin{tabular}{l l}
			Node& $Q_{\textrm{max},t}$ [kW]\\ 
			\midrule
			3&820\\
			5&1925\\
			7&770\\
			9&4025\\
			10&875\\
			12&2205\\
			13&560\\
			\midrule
		\end{tabular}
	\end{center}
\end{table}
\begin{table}
	\caption[Thermal transfer coefficients]{Thermal transfer coefficients $\lambda_s$ for a range of pipe diameters $D_s$}
	\label{table:heattr}
	\begin{center}
		\begin{tabular}{l l}
			$D_s$ [mm]& $\lambda_s$ [W/(m $^\circ$C)]\\
			\midrule
			32&0.189\\
			40&0.210\\
			50&0.219\\
			65&0.236\\
			80&0.278\\
			100&0.327\\
			125&0.321\\
				\midrule
		\end{tabular}
	\end{center}
\end{table}

\begin{table}
	\caption[Operational data]{Operational data}
	\label{table:inputdataobj}
	\begin{center}
		\begin{tabular}{l l}
			Type&Value\\
			\midrule
			Electricity price $c_{\rm elec}$&0.095 \textdollar/kWh\\
			Fuel price $c_{\rm fuel}$&0.07 \textdollar/kWh\\
			Fuel to heat efficiency $\nu_{\rm fuel}$&0.7\\
			Pump efficiency $\nu_{\rm pump}$&0.8\\
			Maximum pump pressure differential&16 bar\\
			Maximum supply temperature&120$^\circ$C\\
			Ground temperature&7$^\circ$C\\
			\midrule
		\end{tabular}
	\end{center}
\end{table}

We evaluate the strategies for design consumer supply temperatures 90$^\circ$C, 100$^\circ$C, 110$^\circ$C and 120$^\circ$C, and for a target pressure loss range between 100 and 1000Pa/m. As in \cite{Pirouti2013}, we assume that the return temperature at the generator is $70^\circ$C. We neglect the thermal losses in the return pipe since they are very small. As this example is a tree network with a single pump, the loop containing the pump and the most remote load governs the pressure profile in the network, and the flow direction in each branch is fixed in advance. Figs.~\ref{fig:pressurefit} and \ref{fig:thermalfit} in the Appendix illustrate the quality of the polynomial pressure drop and thermal loss approximation along all pipe segments for a design temperature of $90^\circ$C and a target pressure loss of $100$Pa/m.

To apply the method we also need to assume a probability distribution for the demand. In steady-state, the energy demand of buildings in the network mostly depends on the ambient temperature \cite{Khosravi2017}. For the winter period, the load roughly varies between $50\%$ and $100\%$ of the maximum load \cite[Fig. 3]{Pirouti2013}. Since we do not have detailed statistical load data for this case, we assume that the heating demand is given by $Q_{t}=Q_{\textrm{max},t}r$, where $Q_{\textrm{max},,t}$ is the maximum load in branch $t$, and $r$ is a latent random variable following a uniform distribution on the interval $[0.5,1]$. 

We optimize the following strategies and compare them in terms of average hourly cost incurred by hydraulic and thermal losses:

\begin{itemize}
	\item VT-VF: The heat flow is regulated by adjusting the supply temperature of the generator and the mass flow rate at node $1$. We approximate the expected value $\mathbf{E}_{\varphi}(\rho(Q_{\rm heat}))$ and the distribution of optimal solutions when (\ref{eq:vtvfoptim}) is solved for each $Q_\textrm{heat}$ in $\mathbf{Y}_L$ using the \ac{SDP} relaxation (\ref{eq:sparsesdp2}), where the product measure constraint (\ref{eq:nonanticip2}) was replaced by 
    \begin{equation}\label{eq:marginalimposition}
    	m_{0 0\gamma}(I_s)=z_{\gamma}(I_s), \quad\forall \gamma \in \mathbb{N}^{p_s},\sum_{t=1}^{p_s} \gamma_t \leq 2k, s=1,\ldots,N_I.
    \end{equation}
   This imposes the moments of $\varphi$ on $\mu$ such that the optimal value of this SDP relaxation approximates $\mathbf{E}_{\varphi}(\rho(Q_{\rm heat}))$ from below \cite{Lasserre2010}.
	\item CT-VF: The heat flow is regulated by adjusting the mass flow rate while the supply temperature of the generator is held constant at node $1$. We approximate problem (\ref{eq:vtcfoptim_1})-(\ref{eq:vtcfoptim_2}) using the \ac{SDP} relaxations (\ref{eq:sparsesdp2}) to determine the optimal supply temperature set-point.
	\item VT-CF: The heat flow is regulated by adjusting the supply temperature while the mass flow in the generator branch is held constant. We approximate problem (\ref{eq:ctvfoptim_1})-(\ref{eq:ctvfoptim_2}) using the \ac{SDP} relaxations (\ref{eq:sparsesdp2}) to determine the optimal mass flow set-point.
\end{itemize}

The main results on expected cost of losses computed by the \ac{SDP} relaxations with $2k=4$ for all design cases and strategies are presented in three figures: the total expected cost $\mathbf{E}_\varphi (\sum_{s\in\mathcal{B}_P}\frac{c_{\rm fuel}}{\nu_{\rm fuel}}Q_s+\sum_{k\in\mathcal{B}_G}c_{\rm elec}P_k)$ is shown in Fig.~\ref{fig:relaxcomp}, and this cost is decomposed into hydraulic and thermal losses in Figs.~\ref{fig:pumpcost} and \ref{fig:thermalcost}.

\begin{figure}
	\centering
	\includegraphics[trim={0 0cm 0 0},clip,scale=1]{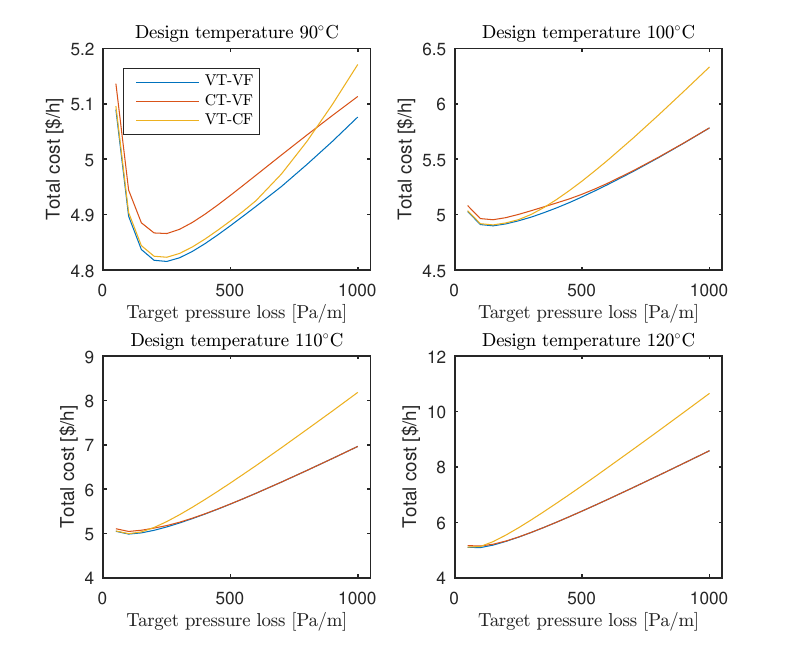}
	\caption[Expected cost of hydraulic and thermal losses]{Expected cost of hydraulic and thermal losses for different design cases computed using \ac{SDP} relaxations with $2k=4$}\label{fig:relaxcomp}
\end{figure}
\begin{figure}
	\centering
	\includegraphics[trim={0 0cm 0 0},clip,scale=1]{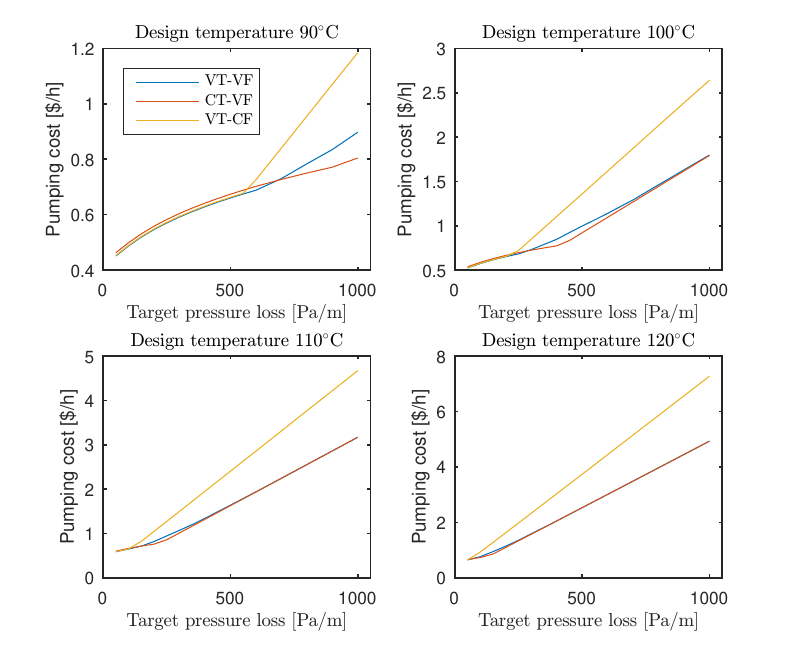}
	\caption[Expected cost of hydraulic losses]{Expected hydraulic cost for different design cases computed using \ac{SDP} relaxations with $2k=4$}\label{fig:pumpcost}
\end{figure}
\begin{figure}
	\centering
	\includegraphics[trim={0 0cm 0 0},clip,scale=1]{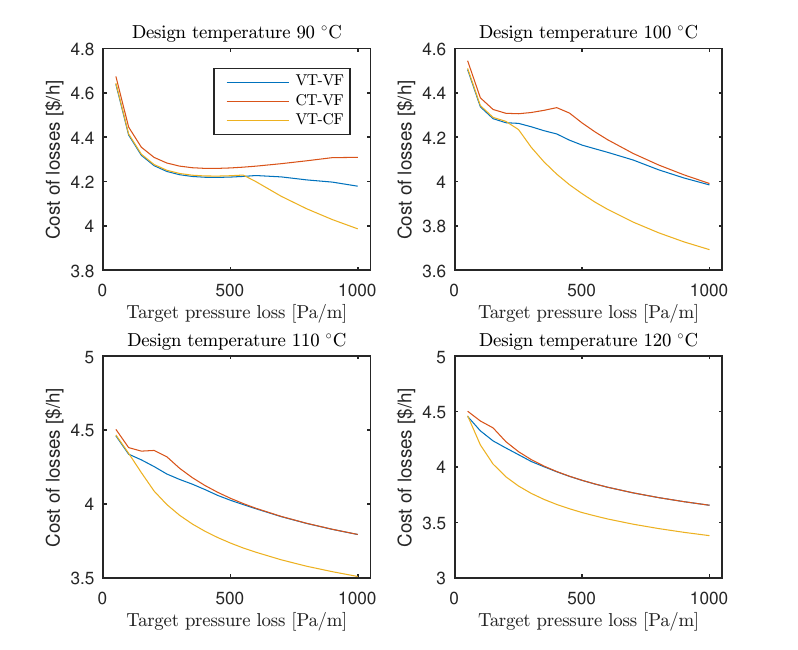}
	\caption[Expected cost of thermal losses]{Expected cost of thermal losses for different design cases computed using \ac{SDP} relaxations with $2k=4$}\label{fig:thermalcost}
\end{figure}

Compared to the VT-VF strategy, the VT-CF strategy only performs well for networks designed with low target pressure loss values (see Fig.~\ref{fig:relaxcomp}). Since the fixed mass flow rate of the VT-CF strategy must be sufficient to provide the maximum load, high hydraulic losses are incurred (see Fig.~\ref{fig:pumpcost}), particularly for networks with small pipe diameters that result from designing with high target pressure loss values and supply temperatures.

The VT-CF strategy performs slightly better than the CT-VF strategy in a limited number of cases with low target pressure loss values. To determine the source of this difference, we investigated the variance of the three operating strategies. This information is contained in the second order moments computed by the relaxations and is reported in Figs.~\ref{fig:massflowrates} and \ref{fig:supplytemperatures}. We note that the VT-CF strategy outperforms the CT-VF strategy in cases for which the optimal VT-VF strategy primarily varies the supply temperature, as indicated by the low mass flow rate and high temperature variance in Figs.~\ref{fig:massflowrates} and \ref{fig:supplytemperatures}. In all other cases, the CT-VF strategy performs better than the VT-CF and almost as well as the VT-VF. The reason for this becomes clear by studying Figs.~\ref{fig:pumpcost} and \ref{fig:thermalcost}. The hydraulic cost increase cannot be compensated by the lower thermal losses observed for higher target pressure loss values. This suggests that lowering the mass flow rate to reduce costly hydraulic losses is more important than lowering the supply temperature to reduce thermal losses. In the worst case observed, the VT-CF strategy is 26.5$\%$ more costly than CT-VF strategy. Note that for a full cost analysis of a strategy change, one would need to consider additional aspects such as the investment cost of variable-speed pumps. To do this, one would need to integrate our approach into the design procedure detailed in \cite{Pirouti2013}.

\begin{figure}
	\centering
	\includegraphics[trim={0 0cm 0 0},clip,scale=1]{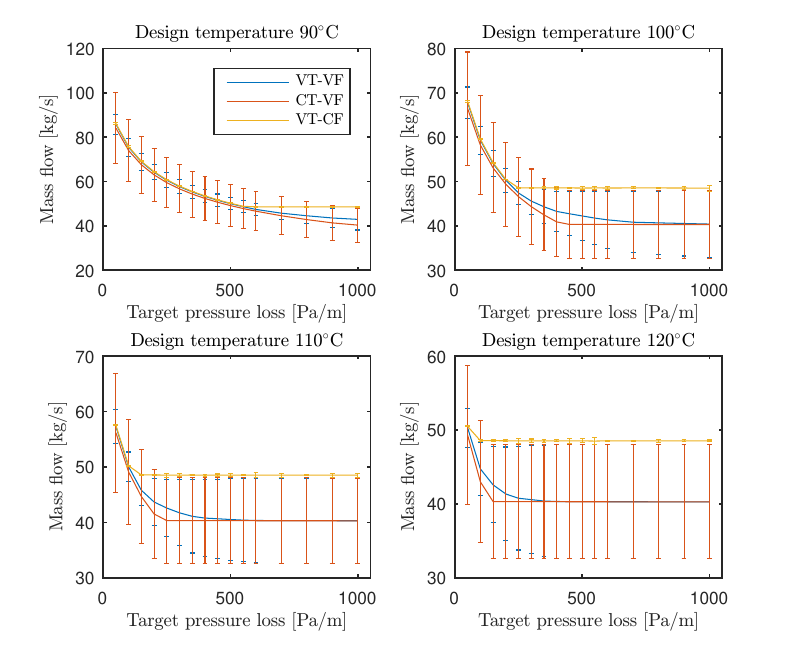}
	\caption[Estimated mean and standard deviation of the distribution of optimal pump mass flow rates]{Estimated mean and standard deviation of the distribution of optimal pump mass flow rates for different design cases computed using \ac{SDP} relaxations with $2k=4$. The second moment is shown in terms of a symmetric standard deviation for illustration purposes. Note that the VT-CF strategy has no flow variance.}\label{fig:massflowrates}
\end{figure}
\begin{figure}
	\centering
	\includegraphics[trim={0 0cm 0 0},clip,scale=1]{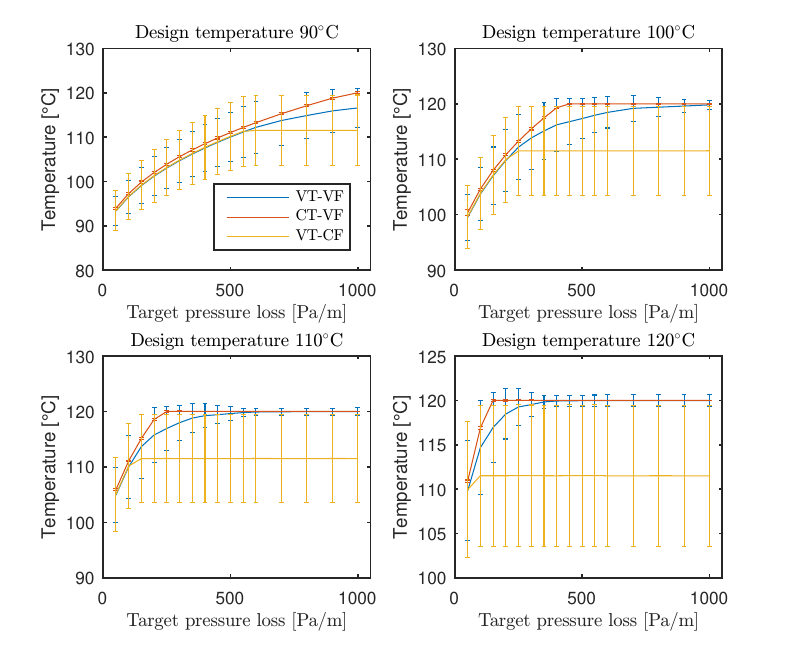}
	\caption[Estimated mean and standard deviation of the distribution of optimal generator supply temperatures]{Estimated mean and standard deviation of the distribution of optimal generator supply temperatures for different design cases computed using \ac{SDP} relaxations with $2k=4$. The second moment is shown in terms of a symmetric standard deviation for illustration purposes. Note that the CT-VF strategy has no temperature variance.}\label{fig:supplytemperatures}
\end{figure}

Increasing the order of the moment relaxation significantly increases the computation time of the resulting optimization problem. It takes between $2.6$ and $8.3$ seconds to solve a single sparse \ac{SDP} relaxations with $2k=4$, and between 136 and 421 seconds to solve a single sparse SDP relaxations with $2k=6$, on a PC with an Intel-i5 2.2GHz CPU with 8GB RAM. The expected total cost of the VT-VF and CT-VF strategies for the next relaxation level $2k=6$ is almost identical to the \ac{SDP} relaxation with $2k=4$, affirming the conclusion that the CT-VF strategy performs almost as well as the VT-VF strategy. In case of the VT-CF strategy, the \ac{SDP} relaxation with $2k=6$ gives a slightly higher expected cost estimate for high target pressure loss values in Fig.~\ref{fig:totalcost_3order}. This supports the previous statement that the VT-CF strategy is a suboptimal choice for networks designed for high design supply temperatures and target pressure loss values. One could check if this trends persists using \ac{SDP} relaxations with $2k=8$ on a PC with more than 8GB RAM to store the significantly larger moment and localizing matrices of size. The moment problems were implemented using YALMIP \cite{Lofberg2004} and solved with MOSEK\textsuperscript{TM}. We used SparsePOP \cite{Waki2008} to detect the sparsity pattern in the problem data.
\begin{figure}
	\centering
	\includegraphics[trim={0 0cm 0 0},clip,scale=1]{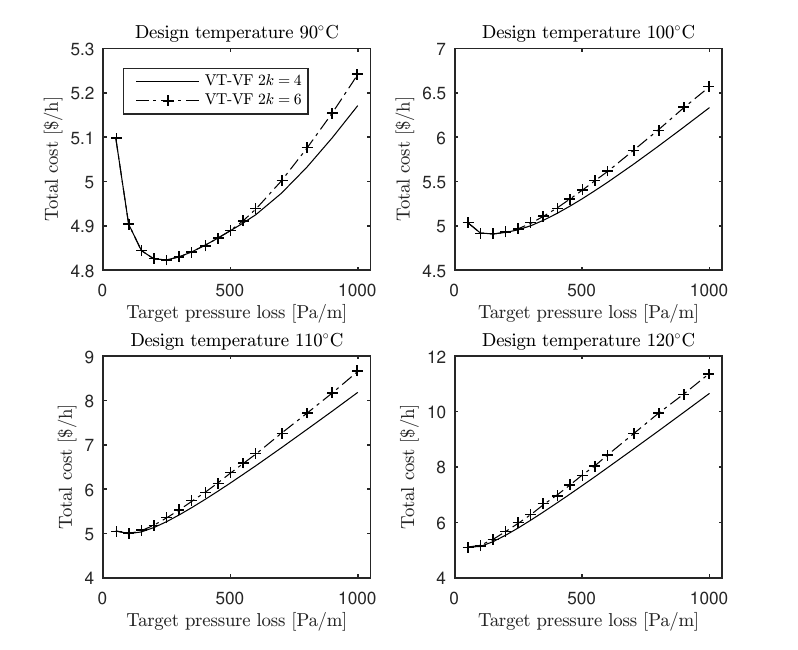}
	\caption[Expected cost of hydraulic and thermal losses using SDP relaxations with $2k=4$ and $2k=6$]{Expected cost of hydraulic and thermal losses of the VT-CF strategy for different design cases computed using \ac{SDP} relaxations with $2k=4$ and $2k=6$. The expected cost of the VT-VF and CT-VF, not shown in this figure, are effectively identical for $2k=4$ and $2k=6$.}\label{fig:totalcost_3order}
\end{figure}

\section{Conclusion}\label{twostagepoly_conclusion}
In this paper, we provided \ac{SDP} relaxations to approximate two-stage stochastic programs with polynomial objective and polynomial constraints. Using our approach, we minimized the expected operating cost of different district heating operational strategies for various design cases. We showed that when optimized in the systematic manner proposed, a strategy that varies the mass flow and holds the supply temperature constant can perform almost as well as one that optimizes both the mass flow and the supply temperature for each load condition.

While this study is focused on the steady-state case, the district heating model as well as our computational framework could be extended to dynamic processes found in long-distance district heating networks (see e.g.~\cite{vanderheijde2017}), in which the rate of change of demand cannot be taken as slow compared to the network transients. For instance, this could be done by adding additional stages to the stochastic program or by using occupation measures \cite{Savorgnan2009}. On a more technical level, the \ac{SDP} relaxations could extend the robust optimization approach of \cite{Lasserre2012} to include recourse, based on recent results in \cite{Jasour2015} and \cite{Lasserre2017}. It is also highly desirable to obtain a rate of convergence for the \ac{SDP} relaxations of Section \ref{twostagemoment}, in the spirit of \cite{korda2016}.

\section*{Acknowledgments}
We would like to thank Viktor Dorer, Roy Smith and Jan Carmeliet for their valuable help and support. We are also grateful to Felix B\"unning, Danhong Wang, Georgios Darivianakis, Benjamin Flamm, Mohammad Khosravi and Annika Eichler for fruitful discussions.
This research project is financially supported by the Swiss Innovation Agency Innosuisse and by NanoTera.ch under the project HeatReserves, and is part of the Swiss Competence Center for Energy Research SCCER FEEB\&D.

\bibliography{ThesisReferences}
\appendix
 \section{Proof of Theorem \ref{thm:gmp_equiv}}\label{app:gmp_equiv}
 \begin{theorem*}
	Two-stage problem (\ref{eq:prefix1})-(\ref{eq:prefix2}) and the \ac{GMP} (\ref{eq:onestage}) are equivalent, in that
\begin{enumerate}
\item[(a)] $\rho_{12}=\rho$, and 
\item[(b)] if $x^*_1$ is an optimal solution of (\ref{eq:prefix1}), then (\ref{eq:onestage}) has an optimal solution that includes the Dirac measure $\mu_1^*=\delta_{x^*_1}$.
\end{enumerate}
 \end{theorem*}
	\begin{proof}
		The proof consists of two steps. We start by stating a \ac{GMP} that is equivalent to the first stage (\ref{eq:prefix1}). We then replace the value function term in the first stage \ac{GMP} by a second stage \ac{GMP} using duality arguments. Consider the following reformulation of (\ref{eq:prefix1}) as a \ac{GMP}: 
		\begin{equation}\label{eq:genfirst}
			\begin{aligned}
			\rho_1=\min_{\mu_1 \in \mathcal{M}(\mathbf{K}_1)_+} &\int_{\mathbf{K}_1} 		f_1d\mu_1+\int_{\mathbf{K}_1}\int_{\mathbf{Y}} v^* d\varphi d\mu_1\\
			\textrm{s.t.}\quad &\int_{\mathbf{K}_1} d\mu_1 =1,\\
			\end{aligned}
		\end{equation}
		where $v^*(x_1,y)$ is the value function of the second stage problem (\ref{eq:prefix2}) and $\int_{\mathbf{Y}} v^* d\varphi=\mathbf{E}_\varphi(v(x_1,y))$ its expected value with respect to $\varphi$. 
        
		
	   By Fubini's Theorem, we can write $\int_{\mathbf{K}_1}\int_{\mathbf{Y}} v^* d\varphi d\mu_1$ as $\int_{\mathbf{K}_1\times\mathbf{Y}} v^* d(\mu_1\otimes\varphi)$. Under Assumption \ref{as:poly2}, the expected value term $\int_{\mathbf{K}_1\times\mathbf{Y}} v^* d(\mu_1\otimes\varphi)$ is the optimal value of an infinite-dimensional \ac{LP} problem over the space of bounded functions $\mathcal{C}(\mathbf{K}_1\times\mathbf{Y})$ \cite[Corollary 2.5]{Lasserre2010}:
			\begin{equation}\label{eq:dualsecond}
			\begin{aligned}		
			d_2=\max_{v \in \mathcal{C}(\mathbf{K}_1\times\mathbf{Y})} &\int_{\mathbf{K}_1\times\mathbf{Y}} v d(\mu_1\otimes\varphi)\\
			\textrm{s.t.}\quad &f_2(x_1,x_2)-v(x_1,y) \geq 0,\quad \forall(x_1,x_2,y)\in\mathbf{K}_2
			\end{aligned}
		\end{equation}
		Replacing $\int_{\mathbf{K}_1\times\mathbf{Y}} v^* d(\mu_1\otimes\varphi)$ in (\ref{eq:genfirst}) by (\ref{eq:dualsecond}) leads to a nonlinear $\min-\max$ problem.
		Instead, we dualize (\ref{eq:dualsecond}) to obtain a \ac{GMP} with fixed marginal measures $\varphi$ and $\mu_1$ \cite{Lasserre2010}:
		\begin{equation}\label{eq:gensecond}
		\begin{aligned}
		\rho_2=&\min_{\mu_2 \in \mathcal{M}(\mathbf{K}_2)_+}\int_{\mathbf{K}_2} f_2d\mu_2\\
		\textrm{s.t.}\quad &\pi_{x_1,y}\mu_2= \mu_1\otimes\varphi
		\end{aligned}
		\end{equation}
		The \ac{GMP} (\ref{eq:gensecond}) encodes all the instances of the optimal second stage solutions $x_2^*(x_1,y)$ in the measure $\mu_2^*$ supported on $\mathbf{K}_2$, where $\mu^*_2$ is the optimal solution of (\ref{eq:gensecond}) given the product measure $\mu_1\otimes\varphi$. By Assumption \ref{as:poly2} and by virtue of \cite[Lemma 2.4]{Lasserre2010}, there is no duality gap between (\ref{eq:dualsecond}) and (\ref{eq:gensecond}). This means that $\int_{\mathbf{K}_1\times\mathbf{Y}} v^* d(\mu_1\otimes\varphi)=\int_{\mathbf{K}_2} f_2d\mu^*_2$.
		Thus, by replacing $\int_{\mathbf{K}_1\times\mathbf{Y}} v^* d(\mu_1\otimes\varphi)$ in (\ref{eq:genfirst}) by (\ref{eq:gensecond}), we obtain:
		\begin{equation}\label{eq:onestageminmin}
			\begin{aligned}
			\rho_{1}=\min_{\mu_1 \in \mathcal{M}(\mathbf{K}_1)_+} &\int_{\mathbf{K}_1} f_1d\mu_1+
			\begin{Bmatrix}
			\min_{\mu_2 \in \mathcal{M}(\mathbf{K}_2)_+}\int_{\mathbf{K}_2} f_2d\mu_2\\
			\enskip\textrm{s.t.}\enskip\pi_{x_1,y}\mu_2=\mu_1\otimes\varphi
			\end{Bmatrix}\\
			&\textrm{s.t.} \quad\int_{\mathbf{K}_1} d\mu_1 =1,\\
			\end{aligned}
		\end{equation}
		By merging the $\min$-operators of (\ref{eq:onestageminmin}), we obtain (\ref{eq:onestage}), with $\rho_1=\rho_{12}$. By virtue of \cite[Theorem 1.1]{Lasserre2014}, we have $\rho_1=\rho$ and the optimal solution $\mu_1^*$ includes a Dirac measure $\delta_{x^*_1}$, where $x^*_1$ is an optimal solution of (\ref{eq:prefix1}).
	\end{proof}
\section{Dual problem}\label{dualproblem}
	In this section, we present an infinite-dimensional linear program over bounded functions that describes the optimal value of the two-stage stochastic program and the second stage value function $v^*(x_1,y)$. This program can be approximated using \ac{SOS} techniques. Consider the following infinite-dimensional \ac{LP} dual to (\ref{eq:genfirst}):
	\begin{equation}
		\begin{aligned}\label{eq:dualfirst}
		d_1=\max_{e \in \mathbb{R}} e\\
		\textrm{s.t.} \quad& f_1(x_1)+\underbrace{\int_\mathbf{Y} v^*(x_1,y)d\varphi}_{\mathbf{E}_\varphi(v^*(x_1,y))}-e \geq 0,\quad\forall x_1\in\mathbf{K}_1.
		\end{aligned}
	\end{equation}
	By \cite[Corollary 1.4]{Lasserre2014}, there is no duality gap between (\ref{eq:genfirst}) and (\ref{eq:dualfirst}), hence $d_1=\rho_1$. Moreover, for each $x_1\in\mathbf{K}_1$, we have:
	\begin{equation}\label{eq:dualsecond2}
		\begin{aligned}		
			\int_\mathbf{Y} v^*(x_1,y)d\varphi=\max_{v \in C({\mathbf{K}_1} \times \mathbf{Y})} \int_{\mathbf{K}_1}\int_\mathbf{Y} v(x_1,y)d\varphi d\delta_{x_1}\\
			\textrm{s.t.}\quad f_2(x_1,x_2)-v(x_1,y) \geq 0, \quad \forall(x_1,x_2,y)\in\mathbf{K}_2.
		\end{aligned}
	\end{equation}
	By inserting problem (\ref{eq:dualsecond2}) into (\ref{eq:dualfirst}) and merging the $\max$-operators, we obtain a single stage problem:
	\begin{equation}\label{eq:onestagedual}
		\begin{aligned}
		d_{12}=\max_{e \in \mathbb{R}, v \in C(\mathbf{K}_1\times \mathbf{Y})} &e\\
		\textrm{s.t.}\quad &f_1(x_1)+\int_\mathbf{Y} v(x_1,y)d\varphi-e \geq 0, \quad\forall x_1\in\mathbf{K}_1,\\
		&f_2(x_2)-v(x_1,y) \geq 0, \quad\forall(x_1,x_2,y)\in\mathbf{K}_2.
		\end{aligned}
	\end{equation}
	\begin{theorem}\label{th:isdualof}
		Under Assumption \ref{as:poly2}, the \ac{LP} (\ref{eq:onestagedual}) over bounded functions is the dual of the \ac{GMP} (\ref{eq:onestage}) and the duality gap between (\ref{eq:onestage}) and (\ref{eq:onestagedual}) is zero.
	\end{theorem}
	\begin{proof}
	
	To show that (\ref{eq:onestagedual}) is the dual of (\ref{eq:onestage}), we state (\ref{eq:onestage}) as a linear program in canonical form \cite{barvinok2002}:
	\begin{equation}\label{eq:canprimal}
	\begin{aligned}
	p=\min_x \langle x,c \rangle_1\\
	\textrm{s.t.}\quad \mathcal{A}(x)=b\\
	x \in C
	\end{aligned}
	\end{equation}	
	where $\langle\rangle_1: E_1\times F_1\rightarrow\mathbb{R}$ is a duality of vector spaces, $\mathcal{A}$ a linear map $\mathcal{A}: E_1 \rightarrow E_2$, $E_2$ being another vector space, and $C$ is a convex cone $C\subset E_1$. We have $x\in E_1$, $b\in E_2$ and $c\in F_1$.
	
	The canonical dual form is given as:
	\begin{equation}\label{eq:candual}
	\begin{aligned}
	d=\max_y &\langle b,l \rangle_2\\
	\textrm{s.t.}\quad &c-\mathcal{A}'l \in C',\\
	&y\in F_2\\
	\end{aligned}
	\end{equation}
	where $\langle\rangle_2: E_2\times F_2\rightarrow\mathbb{R}$ is a duality of vector spaces, and $\mathcal{A'}$ the adjoint linear map $\mathcal{A'}: F_2 \rightarrow F_1$ such that
	\begin{equation}\label{eq:adjointdef}
	\langle \mathcal{A}(x),l \rangle = \langle x,\mathcal{A}'(l) \rangle \quad\forall x \in E_1,\forall l \in F_1.
	\end{equation}
	The convex cone $C'$ is the dual cone $C'\subset F_1$ defined by 
	\begin{equation}
	C'=\{l\in F_1: \langle x,l\rangle\geq 0 \quad\forall x\in C \}.
	\end{equation}
	
	In Table \ref{table:canonical}, we establish the correspondence between the canonical form and the two-stage problem.
	\begin{table}
		\caption[Definitions of the elements of the canonical primal and dual linear programs]{Definitions of the elements of the canonical primal (\ref{eq:canprimal}) and dual (\ref{eq:candual}) \acp{LP}}
		\label{table:canonical}
		\begin{center}
			\begin{tabular}{ l l l}
				Type& Canonical & Two-stage problem\\
				\midrule
				Vector spaces  & $E_1$ &$\mathcal{M}(\mathbf{K}_1)\times\mathcal{M}(\mathbf{K}_2)$\\
				& $E_2$ &$\mathbb{R}\times\mathcal{M}(\mathbf{K}_1\times\mathbf{Y})$\\
				& $F_1$ &$\mathcal{C}(\mathbf{K}_1)\times\mathcal{C}(\mathbf{K}_2)$\\
				& $F_2$ &$\mathbb{R}\times\mathcal{C}(\mathbf{K}_1\times\mathbf{Y})$\\
				\midrule
				Variables & $x\in E_1$ & $(\mu_1,\mu_2)\in\mathcal{M}(\mathbf{K}_1)\times\mathcal{M}(\mathbf{K}_2)$\\
				& $l\in F_2$ & $(e,v)\in\mathbb{R}\times\mathcal{C}(\mathbf{K}_1\times\mathbf{Y})$\\
				\midrule
				Problem data & $b\in E_2$ & $(1,0)\in\mathbb{R}\times\mathcal{M}(\mathbf{K}_1\times\mathbf{Y})$\\
				& $c\in F_1$ & $(f_1,f_2)\in\mathcal{C}(\mathbf{K}_1)\times\mathcal{C}(\mathbf{K}_2)$\\
				\midrule
				Duality pairings & $\langle\rangle_1: E_1\times F_1\rightarrow\mathbb{R}$ & $\langle (\mu_1,\mu_2),(f_1,f_2) \rangle_1 =$\\
				& &$\quad \int_{\mathbf{K}_1} f_1d\mu_1+\int_{\mathbf{K}_2} f_2d\mu_2$\\
				& $\langle\rangle_2: E_2\times F_2\rightarrow\mathbb{R}$ & $\langle (1,0),(e,v)\rangle_2=e$\\
				\midrule				 
				Convex cone and dual & $C\subset E_1$ & $\mathcal{M}(\mathbf{K}_1)_+\times\mathcal{M}(\mathbf{K}_2)_+$\\
				& $C'\subset F_1$ & $\mathcal{C}(\mathbf{K}_1)_+\times\mathcal{C}(\mathbf{K}_2)_+$\\		 
				\bottomrule
			\end{tabular}
		\end{center}
	\end{table}
	
	
	The linear operator is constructed from (\ref{eq:onestage}):
	\begin{equation}
	\mathcal{A}(\mu_1,\mu_2)=\begin{bmatrix}
	\int_{\mathbf{K}_1}d\mu_1\\
	\pi_{x_1,y}\mu_2- \mu_1\otimes\varphi 
	\end{bmatrix}\\
	\end{equation}
	Let $\pi_{x_1,y}':\mathcal{C}(\mathbf{K}_2) \rightarrow \mathcal{C}(\mathbf{K})$ be the adjoint operator of $\pi_{x_1,y}$ defined by $\pi'_{x_1,y}: (x_1,y)\mapsto (\pi'v)(x_1,x_2,y)=v(x_1,y)$. Using the definition (\ref{eq:adjointdef}) of the adjoint operator, we obtain:	
	\begin{equation}
	\mathcal{A}'(e,v)=\begin{bmatrix}
	e-\int_\mathbf{Y} \pi_{x_1,y}'v \medspace d\varphi\\
	\pi_{x_1,y}'v
	\end{bmatrix}
	\end{equation}
	Inserting the definitions of Table \ref{table:canonical} and the adjoint operator $\mathcal{A}'(e,v)$ in (\ref{eq:candual}), we obtain the dual formulation (\ref{eq:onestagedual}).
    
	If the cone
	\begin{equation*}
	D=\{(\mathcal{A}\mu,\langle f,\mu\rangle): \mu \in \mathcal{M}(\mathbf{K}_1)_+\times\mathcal{M}(\mathbf{K}_2)_+\}
	\end{equation*}
	is closed in $(\mathbb{R} \times \mathcal{M}(\mathbf{K}_1 \times \mathbf{Y})) \times \mathbb{R}$ for $\mu=(\mu_1,\mu_2)$ and $f=(f_1,f_2)$, there is no duality gap, thus $\rho_{12}=d_{12}$ \cite{barvinok2002}. Following the arguments of \cite[C.4]{Lasserre2014}, we start with a sequence $\mu_n$ such that 
	\begin{equation}\label{eq:conv_sa}
	\lim_{n \to \infty}(\mathcal{A}\mu_n,\langle f,\mu_n\rangle)=(a,s) 
	\end{equation}
	for some $(a,s)\in(\mathbb{R} \times \mathcal{M}(\mathbf{K}_1 \times \mathbf{Y})) \times \mathbb{R}$. Taking a test function $(1,1)$, we observe that the sequence is bounded because
	\begin{equation}
	\begin{aligned}
	&\langle\mathcal{A}\mu_n,(1,1)\rangle=\\
	&\quad\quad\mu_{1,n}(\mathbf{K}_1) +\underbrace{\pi_{x_{1},y}\mu_{2,n}(  \mathbf{K}_1\times\mathbf{Y})}_{=\mu_{2,n}(\mathbf{K}_2)}+\mu_{1,n}(\mathbf{K}_1)\varphi(\mathbf{Y}) \rightarrow a <\infty
	\end{aligned}
	\end{equation}
	and the measures are non-negative. By the Alaoglu Theorem \cite[Section 5.10, Theorem 1]{luenberger1970}, there is a subsequence $\mu_{n_k}$ that converges weakly to $\mu$. Using this fact, relationship (\ref{eq:conv_sa}) and the continuity of $\mathcal{A}$, we have $(a,s)=(\mathcal{A}\mu,\langle f,\mu\rangle)$. This means that $D$ is closed. 
\end{proof}    

\section{Hydraulic loss approximation}\label{app:pressureloss}
The pressure drop $\Delta p_s$ in [Pa] along a pipe segment $s\in\mathcal{B}_P$ is modelled as \cite{Liu2013a}
\begin{equation}\label{eq:pressuredrop}
\Delta p_s=K_{s}|m_{s}|m_{s}, \quad\forall s\in\mathcal{B}_{P},
\end{equation}
where $K_{s}$ is the resistance coefficient
\begin{equation}\label{eq:frictioncoeff}
K_{s}=\frac{8L_{s}}{D_{s}^5\rho\pi^2}f_{s},
\end{equation}
where $D_{s}$ is the pipe diameter in [m], $L_{s}$ the pipe length in [m], $\rho$ the water density in [kg/m$^3$] and $f_{s}$ the friction factor. 
Let $Re_s$ be the Reynolds number given by
\begin{equation}
Re_{s}=\frac{m_{s}}{\mu\rho\pi D_{s}/4},
\end{equation}
where $\mu$ is the kinematic viscosity of water in [m$^2$/s]. In a turbulent regime, we have $Re_{s}>4000$ \cite{Liu2013a}. Assuming a turbulent flow regime, the friction factor $f_{s}$ is governed by the equation of Colebrook-White \cite{Wang2018,Liu2013a}
\begin{equation}\label{eq:friction}
\frac{1}{\sqrt{f_{s}}}=-2\log_{10}\Bigg(\frac{\epsilon_{s}/D_{s}}{3.7}+\frac{2.51}{Re_{s}\sqrt{f_{s}}}\Bigg),
\end{equation}
where $\epsilon_{s}$ is the roughness of the pipe in [m]. Equation (\ref{eq:friction}) cannot be incorporated in a polynomial optimization framework. However, the friction factor $f_{s}$ changes only moderately and is often assumed to be constant (e.g.~\cite{AWAD2009,Tang2014,vanderheijde2017}). Thus, we solve equations (\ref{eq:pressuredrop})-(\ref{eq:friction}) for a range of mass flow rates $[0,m_{\textrm{max},s}]$ and fit a second order polynomial function to $\Delta p_s$.
\begin{figure}
	\centering
	\includegraphics[trim={0 0cm 0 0},clip,scale=1]{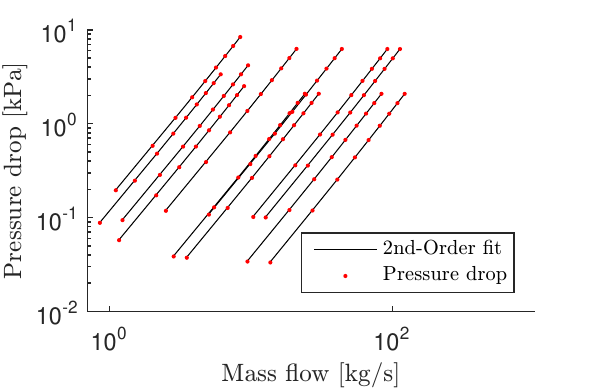}
	\caption[2nd-Order polynomial fit of the hydraulic pressure drop]{2nd-Order polynomial fit (black line), as defined in (\ref{eq:pressuredrop_fitted}), of the hydraulic pressure drop (red dots), as computed by equations (\ref{eq:pressuredrop})-(\ref{eq:friction}), along all pipe elements designed with $\hat{T}_{\textrm{out},s}=90^\circ$C and a target pressure of $100 $Pa/m. Both axes are on a logarithmic scale.}\label{fig:pressurefit}
\end{figure}%
\section{Thermal loss approximation}\label{app:thermalloss}
The outlet temperature $T_{\textrm{out},s}$ of a pipe branch is given by \cite{Liu2014}:
\begin{equation}\label{eq:thermalloss}
T_{\textrm{out},s}=(T_{\textrm{in},s}-T_{\textrm{a}})e^{-\frac{L_s\lambda_s}{c_p m_{s}}}+T_{\textrm{amb}},
\end{equation}
where $L_s$ is the pipe length in [m], $\lambda_s$ the per unit length heat transfer coefficient of the pipe material in [W/($^\circ$C m)], $c_p$ the specific heat capacity of water in [J/(kg $^\circ$C)] and $T_{\textrm{amb}}$ the ground temperature in [$^\circ$C]. The exponential expression in (\ref{eq:thermalloss}) cannot be incorporated in a polynomial framework. As suggested in \cite{AWAD2009} and \cite{Tang2014}, we approximate the exponential $e^{-\frac{L_s\lambda_s}{c_p} y}$ in (\ref{eq:thermalloss}), where and $y=\frac{1}{m_{s}}$, by the first two terms of its Taylor series $e^{-\frac{L_s\lambda_s}{c_p} a}-\frac{L_s\lambda_s}{c_p}e^{-\frac{L_s\lambda_s}{c_p} a}(y-a)$ around a point $a=m_{\rm max}/2$. Inserting the first order approximation in (\ref{eq:thermalloss}) and multiplying both sides by $m_{s}$, we obtain the polynomial equation
\begin{equation}\label{eq:polynoimalthermalloss}
T_{\text{out},s}m_{s}=(T_{\text{in},s}-T_{\text{amb}})\Bigg(\Big(\underbrace{e^{-\frac{L_s\lambda_s}{c_p} a}+a\frac{L_s\lambda_s}{c_p}e^{-\frac{L_s\lambda_s}{c_p} a}}_{c_s}\Big)m_s-\underbrace{\frac{L_s\lambda_s}{c_p}e^{-\frac{L_s\lambda_s}{c_p} a}}_{d_s}\Bigg)+T_{\text{amb}}m_{s}.
\end{equation}
\begin{figure}
	\centering
	\includegraphics[trim={0 0cm 0 0},clip,scale=1]{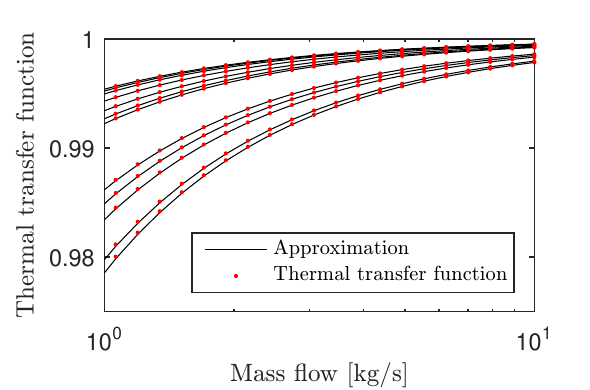}
	\caption[Approximation of thermal transfer function]{Approximation $c_s-\frac{d_s}{m_s}$, as defined in \eqref{eq:polynoimalthermalloss}, of the thermal transfer function $e^{-\frac{L_s\lambda_s}{c_p m_{s}}}$, as defined in \eqref{eq:thermalloss}, for all pipes of a network designed with $\hat{T}_{\textrm{out},s}=90^\circ$C and a target pressure of $100$Pa/m. The horizontal axis is on a logarithmic scale.}\label{fig:thermalfit}
\end{figure}

\end{document}